\documentclass[12pt]{article}
\usepackage{amssymb,amsmath,amsthm,tikz,multirow}
\usetikzlibrary{arrows,calc}

\title{Combinatorial Tilings of the Sphere by Pentagons}
\author{Min Yan\thanks{Research was supported by Hong Kong RGC General Research Fund 605610 and 606311.} \\ Hong Kong University of Science and Technology}

\newtheorem{theorem}{Theorem}
\newtheorem{lemma}[theorem]{Lemma}

\theoremstyle{definition}
\newtheorem*{definition*}{Definition}
\newtheorem*{case*}{Case}
\newtheorem*{subcase*}{Subcase}

\theoremstyle{remark}

\numberwithin{equation}{section}

\begin{document}

\maketitle

\begin{abstract}
A combinatorial tiling of the sphere is naturally given by an embedded graph. We study the case that each tile has exactly five edges, with the ultimate goal of classifying combinatorial tilings of the sphere by geometrically congruent pentagons. We show that the tiling cannot have only one vertex of degree $>3$. Moreover, we construct earth map tilings, which give classifications under the condition that vertices of degree $>3$ are at least of distance $4$ apart, or under the condition that there are exactly two vertices of degree $>3$. 
\end{abstract}

\section{Introduction}

Tilings are usually understood to be composed of geometrically congruent tiles. Typically we have finitely many \emph{prototiles} and require that every tile to be isometric to one of the prototiles. For tilings of the plane by polygons, this means that there is a one-to-one correspondence between the edges of each tile and the edges of a prototile, such that the adjacency of edges is preserved, the edge length is preserved, and the angle between adjacent edges is preserved. 

The combinatorial aspect of the tiling ignores the geometric information. For tilings of the plane, this means that we ignore the edge length and the angle information. What remains is only the number of edges of each tile. Dress, Delgado-Friedrichs, Huson~\cite{del, ddh, dress1, dress2} used the Delaney symbol to encode the combinatorial information and solved many combinatorial tiling problems. The Delaney symbol can also be turned into a computer algorithm for enumerating various classes of combinatorial tilings. For another research direction on the combinatorial aspects of tiling, see Schulte~\cite{s2}.

Our interest in the combinatorial tiling arises from the classification of \emph{edge-to-edge} and \emph{monohedral} tilings of the sphere. The tiles in such a tiling are all congruent and must be triangles, quadrilaterals, or pentagons. The classification of the triangular tilings of the sphere was started by Sommerville~\cite{so} in 1923 and completed by Ueno and Agaoka~\cite{ua} in 2002. We are particularly interested in the pentagonal tilings, which we believe is relatively easier to study than the quadrilateral tilings because $5$ is the ``other extreme'' among $3,4,5$. In \cite{gsy}, we classified the tilings of the sphere by $12$ congruent pentagons, where $12$ is the minimal number of pentagonal tiles. Unlike the triangle case, where the congruence in terms of the edge length is equivalent to the congruence in terms of the angle, we needed to study different kinds of congruences separately, and then obtained the final classification by combining the classifications of different congruences. For $12$ pentagonal tiles, the combinatorial structure is always the dodecahedron. Then we found all $8$ families of edge congruent tilings for the dodecahedron. We also found $7$ families of angle congruent tilings for the dodecahedron, plus perhaps around $20$ families of angle congruent tilings for a remaining configuration of angles in the pentagon. This remaining case is yet to be completely classified, but is fortunately not needed for the final classification in \cite{gsy}.

The purpose of this paper is to study the possible combinatorial configurations beyond the minimal case of $12$ pentagonal tiles. Of course this is only the first step in the complete classification. The next step is to separately study the edge congruence and the angle congruence. The final step is to combine the two congruences together. In \cite{ccy}, we further study edge congruent tilings, and completely classify for the case of the earth map tilings constructed in this paper. In \cite{ccy2}, we further study geometrically congruent tilings, especially for the case of the earth map tilings. On the other hand, in \cite{luk,ly}, we further study the numerics in angle congruent tilings. 

Now we make precise our object of study. A tiling in this paper is naturally given by a graph embedded in the sphere, which divides the sphere into tiles that are homeomorphic to the disk. If the tiles are geometrically congruent pentagons, then the tiling has the following combinatorial property.

\begin{definition*}
A \emph{combinatorial pentagonal tiling of the sphere} is a graph embedded in the sphere, such that the boundary of any tile is a simple closed path consisting of fives edges, and the degree of any vertex is $\ge 3$. 
\end{definition*}

A simple closed path of a combinatorial pentagonal tiling of the sphere divides the sphere into two disks. Each disk has their own combinatorial pentagonal tilings in the following sense.

\begin{definition*}
A \emph{combinatorial pentagonal tiling of the $2$-disk} is a graph embedded in the disk, such that the boundary of the disk consists of some edges of the graph, the boundary of any tile is a simple closed path consisting of fives edges, and the degree of any vertex in the \emph{interior} of the disk is $\ge 3$. 
\end{definition*}

The definition allows some vertices on the boundary to have degree $2$. If the boundary consists of $m$ edges, we also call the tiling a combinatorial pentagonal tiling of the $m$-gon.

Let $v_i$ be the number of vertices of degree $i$. Then it follows from the Euler equation and the Dehn-Sommerville equations that (see \cite[page 750]{gsy}, for example)
\[
v_3 = 20 + 2v_4+5v_5+8v_6 +\dotsb=20+\sum_{i\ge 4}(3i-10)v_i,
\]
and the number of tiles is
\[
f=12+2\sum_{i\ge 4}(i-3)v_i.
\]
In particular, the number of tiles must be even. 

Our first result concerns the smallest number of tiles beyond the minimum $12$. In other words, is there a tiling of the sphere by $14$ pentagons? Theorem \ref{1vertex} says that the tiling must have at least two vertices of degree $>3$, so that the next minimum number should be $16$. In fact, we will see below that there is a unique combinatorial spherical tiling by $16$ pentagons.

The equality above shows that vertices of degree $3$ dominate the others. However, Theorems \ref{distance5} and \ref{distance4} say that vertices of degree $>3$ cannot be too isolated. In fact, any vertex of degree $>3$ must have another vertex of degree $>3$ within distance $5$. Moreover, if the distances between vertices of degree $>3$ are always $\ge 4$, then the tiling must be the earth map tilings in Figures \ref{map5} and \ref{map4}.

On the other hand, vertices of degree $>3$ can also be quite ``crowded'' somewhere in the tiling. For example, given any two combinatorial pentagonal tilings $T$ and $T'$ of the sphere, we may construct the ``connected sum'' $T\# T'$ by deleting one tile each from $T$ and $T'$ and gluing the two together along the five boundary edges. The connected sum is a pentagonal tiling with all vertices along the five boundary edges having degree $>3$.

In general, we may define the \emph{earth map tilings} to be the ones with exactly two vertices of degree $>3$. Naturally we call these two vertices poles. Theorem \ref{2vertex} says that for each distance between the poles ranging from $1$ to $5$, there is exactly one family of earth map tilings. In increasing order of the distance, these families are given by Figures \ref{map1}, \ref{map2}, \ref{map3}, \ref{map4} and \ref{map5}.

Finally, we observe that the number of tiles in an earth map tiling is a multiple of $4$ for distance $5$ and a multiple of $12$ for the other distances. Therefore the only combinatorial tiling by $16$ pentagons is the earth map tiling of distance $5$.

\section{One Vertex of Degree $>3$}

\begin{theorem}\label{1vertex}
There is no combinatorial pentagonal tiling of the sphere with only one vertex of degree $>3$.
\end{theorem}

Let $P$ be a tile with the only vertex of degree $>3$ as one of its five vertices. Then the complement of $P$ is a combinatorial pentagonal tiling of pentagon, such that the only vertex of degree $>3$ is on the boundary. The following result implies that there is no such tiling.

\begin{lemma}\label{cycle}
If a combinatorial pentagonal tiling of the $m$-gon, $m\le 7$, has at most one vertex of degree $>3$, and all vertices in the interior have degree $3$, then the tiling either consists of only one tile, or is the complement of one tile in the dodecahedron tiling.
\end{lemma}

There is a very good reason why we stop at $7$. Figure \ref{89gon} gives combinatorial pentagonal tilings of the $8$-gon and the $9$-gon. Their complements in the dodecahedron tiling are also combinatorial pentagonal tilings. 

\begin{figure}[h]
\centering
\begin{tikzpicture}[>=latex]

\foreach \x in {1,-1}
\draw[xscale=\x]
	(0,0.3) -- (40:0.8) -- (1,0) -- (-40:0.8) -- (0,-0.3) -- cycle;

\foreach \x in {0,1,2}
\draw[xshift=4cm, rotate=120*\x]
	(0,0) -- (30:0.6) -- (0.35,1) -- (-0.35,1) -- (150:0.6); 

\end{tikzpicture}
\caption{Pentagonal tilings of $8$-gon and $9$-gon.}
\label{89gon}
\end{figure}
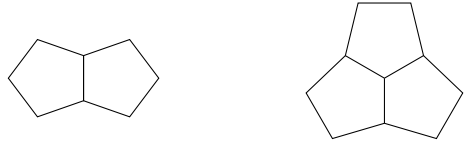

A key technique we use for proving Lemma \ref{cycle} and later results is the criterion for constructing pentagonal tiles. Let $e_-,e,e_+$ be three successive edges, connecting $x,y,z,w$. We say that $e_-$ and $e_+$ are \emph{on the same side} of $e$ if there is a path $\hat{e}$ connecting $x$ and $w$, such that $e_-,e,e_+,\hat{e}$ form a simple closed path, and the region enclosed by the simple closed path does not contain other edges at $y$ and $z$. In other words, all the edges at $y$ and $z$ other than $e_-,e,e_+$ are outside of the region enclosed by the simple closed path. See left of Figure \ref{side}.

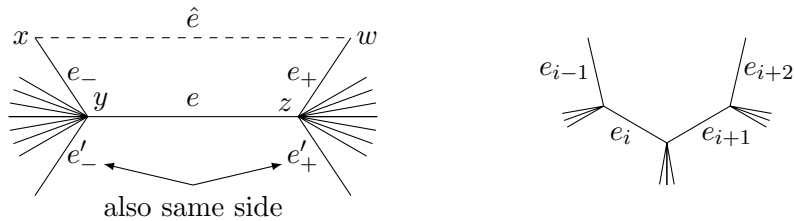
\begin{figure}[h]
\centering
\begin{tikzpicture}[>=latex,scale=0.7]

\draw
	(-2,0) node[above right=-2] {\small $y$} -- node[above] {\small $e$} (2,0) node[above left=-2] {\small $z$}
	(-3,1.5) node[left=-2] {\small $x$} -- node[right=-2] {\small $e_-$} (-2,0) -- node[right=-2] {\small $e_-'$} (-3,-1.5)
	(3,1.5) node[right=-2] {\small $w$} -- node[left=-2] {\small $e_+$} (2,0) -- node[left=-2] {\small $e_+'$} (3,-1.5);
	
\draw[dashed]
	(-3,1.5) -- node[above] {\small $\hat{e}$} (3,1.5);

\draw[<->]
	(-1.7,-0.9) -- (0,-1.3) node[below] {\small also same side} -- (1.7,-0.9);
	
\foreach \x in {1,-1}
\foreach \y in {1,-1}
\draw[xscale=\x, yscale=\y]
	(2,0) -- ++(0:1.5)
	(2,0) -- ++(10:1.5)
	(2,0) -- ++(20:1.5)
	(2,0) -- ++(30:1.5);

\draw[shift={(7.5cm,0.5cm)}]
	(0,1) -- node[left=-2] {\small $e_{i-1}$}
	(0.3,-0.3) -- node[below left=-3] {\small $e_i$}
	(1.5,-1) -- node[below right=-3] {\small $e_{i+1}$}
	(2.7,-0.3) -- node[right=-2] {\small $e_{i+2}$}
	(3,1);
	
\foreach \x in {1,-1}
\draw[shift={(9cm,0.5cm)},xscale=\x]
	(1.2,-0.3) -- ++(-10:0.8)
	(1.2,-0.3) -- ++(-20:0.8)
	(1.2,-0.3) -- ++(-30:0.8)
	(0,-1) -- ++(-80:0.8)
	(0,-1) -- ++(-90:0.8);

\end{tikzpicture}
\caption{Edges on the same side.}
\label{side}
\end{figure}

\begin{lemma}\label{sequence}
Suppose a sequence of edges $e_1,e_2,\dotsc,e_k$ satisfies the following.
\begin{enumerate}
\item $e_i$ and $e_{i+1}$ share one vertex.
\item $e_{i-1}$ and $e_{i+1}$ are on the same side of $e_i$.
\end{enumerate}
Then the edges in the sequence belong to the same tile. In particular, we have $k\le 5$. 
\end{lemma}

The conditions of the lemma are described on the right of Figure \ref{side}. In a combinatorial tiling of the sphere, any edge is shared by exactly two tiles. The tile on the side of $e_i$ that includes the corner between $e_{i-1}$ and $e_i$ is the same as the tile on the side of $e_{i+1}$ that includes the corner between $e_i$ and $e_{i+1}$. Lemma \ref{sequence} follows from this observation.

A consequence of Lemma \ref{sequence} is that the configurations in Figure \ref{impossible} are impossible. For example, if we apply the lemma to the triangle, then we find that the three edges are the boundary edges of a pentagonal tile. However, all five boundary edges of this tile should form a simple closed path. Since three of the five edges already form a simple closed path, we get a contradiction. For the pentagon, we need $n\ge 1$, which means at least one edge pointing to the interior. Beyond the pentagon, the contradiction is to have a tile with more than five boundary edges.

\begin{figure}[h]
\centering
\begin{tikzpicture}[>=latex,scale=1]

\foreach \a in {0,...,5}
\draw[xshift=2.5*\a cm]
	(0,0.8) -- ++(-80:0.3)
	(0,0.8) -- ++(-90:0.3) 
	(0,0.8) -- ++(-100:0.3);

\foreach \a in {0,1,2}
\fill (14+0.3*\a,0) circle (0.05);


\draw
	(0,0) circle (0.8)
	(0,0.8) -- ++(80:0.3)
	(0,0.8) -- ++(90:0.3) 
	(0,0.8) -- ++(100:0.3)
	(0,-0.8) -- ++(-80:0.3)
	(0,-0.8) -- ++(-90:0.3) 
	(0,-0.8) -- ++(-100:0.3);


\begin{scope}[shift={(2.5cm,-0.2cm)}]

\foreach \a in {0,1,2}
{
\draw[rotate=120*\a]
	(-30:1) -- (90:1);
\draw[shift={(90+120*\a:1)}, rotate=120*\a]
	(0,0) -- ++(80:0.3)
	(0,0) -- ++(90:0.3) 
	(0,0) -- ++(100:0.3);
}

\end{scope}


\begin{scope}[xshift=5cm]

\foreach \a in {0,...,3}
{
\draw[rotate=90*\a]
	(0:0.8) -- (90:0.8);
\draw[shift={(90*\a:0.8)}, rotate=90*\a]
	(0,0) -- ++(10:0.3)
	(0,0) -- ++(0:0.3) 
	(0,0) -- ++(-10:0.3);
}

\end{scope}


\begin{scope}[shift={(7.5 cm,-0.1 cm)}]

\foreach \a in {0,...,4}
{
\draw[rotate=72*\a]
	(18:0.9) -- (90:0.9);
\draw[shift={(18+72*\a:0.9)}, rotate=72*\a]
	(0,0) -- ++(8:0.3)
	(0,0) -- ++(18:0.3) 
	(0,0) -- ++(28:0.3);
}

\node at (0,0.45) {\scriptsize $n\ge 1$};
	
\end{scope}


\begin{scope}[xshift=10 cm]

\foreach \a in {0,...,5}
{
\draw[rotate=60*\a]
	(-30:0.8) -- (30:0.8);
\draw[shift={(30+60*\a:0.8)}, rotate=60*\a]
	(0,0) -- ++(20:0.3)
	(0,0) -- ++(30:0.3) 
	(0,0) -- ++(40:0.3);
}

\draw[shift={(150:0.8)}, rotate=-30]
	(0,0) -- ++(10:0.3)
	(0,0) -- ++(0:0.3) 
	(0,0) -- ++(-10:0.3);
	
\end{scope}


\begin{scope}[xshift=12.5 cm]

\foreach \a in {0,...,6}
{
\draw[rotate=51.43*\a]
	(90:0.8) -- (141.43:0.8);
\draw[shift={(90+51.43*\a:0.8)}, rotate=51.43*\a]
	(0,0) -- ++(80:0.3)
	(0,0) -- ++(90:0.3) 
	(0,0) -- ++(100:0.3);
}

\draw[shift={(141.43:0.8)}, rotate=-38.57]
	(0,0) -- ++(10:0.3)
	(0,0) -- ++(0:0.3) 
	(0,0) -- ++(-10:0.3);

\draw[shift={(38.57:0.8)}, rotate=-141.43]
	(0,0) -- ++(10:0.3)
	(0,0) -- ++(0:0.3) 
	(0,0) -- ++(-10:0.3);
			
\end{scope}

\end{tikzpicture}
\caption{Impossible configurations.}
\label{impossible}
\end{figure}
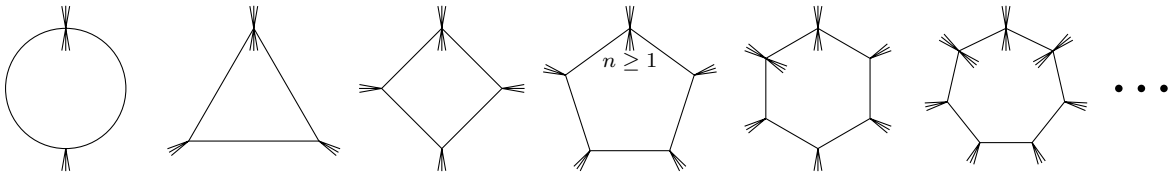

\begin{proof}[Proof of Lemma \ref{cycle}]
Let us first restrict to $m\le 5$, which is already sufficient for proving Theorem \ref{1vertex}. 

Under the given assumption, all the possible boundary configurations of the disk (with $m\le 5$) are listed in Figure \ref{allcase}. The only possible vertex of degree $>3$, which we will call the \emph{high degree vertex}, has degree $n+2$ in the disk.  Here ``high'' only means that there is no limit on how large $n$ can be, and the vertex can also have degree $2$ or $3$ on the other extreme. The lemma basically says that, with the exceptions of (5.1) with $n=0$ and (5.10) with $n=1$, all the other configurations will lead to contradiction.

\begin{figure}[h]
\centering
\begin{tikzpicture}[>=latex,scale=1]


\foreach \x in {0,1}
\draw[shift={(2.5*\x cm,0.2cm)}]
	(0,0) circle (0.8)
	(0,0.8) -- ++(-80:0.3)
	(0,0.8) -- ++(-90:0.3) node[below=-2] {\small $n$}
	(0,0.8) -- ++(-100:0.3);


\node at (0,-0.9) {\small (2.1)};
\fill (0,-0.6) circle (0.05);


\draw (2.5,-0.6) -- ++(90:0.3);
\node at (2.5,-0.9) {\small (2.2)};


\foreach \x in {2,3,4}
\draw[xshift=2.5*\x cm]
	(-30:1) -- (90:1) -- (210:1) -- cycle
	(90:1) -- ++(-80:0.3)
	(90:1) -- ++(-90:0.3) node[below=-2] {\small $n$}
	(90:1) -- ++(-100:0.3);
	

\node at (5,-0.9) {\small (3.1)};


\draw[xshift=7.5cm]
	(-30:1) -- ++(-30:-0.3);
\node at (7.5,-0.9) {\small (3.2)};


\draw[xshift=10cm]
	(-30:1) -- ++(-30:-0.3)
	(210:1) -- ++(210:-0.3);
\node at (10,-0.9) {\small (3.3)};


\begin{scope}[yshift=-2.2cm]

\foreach \x in {0,1,...,5}
\draw[xshift=2*\x cm]
	(-0.8,-0.8) rectangle (0.8,0.8)
	(-0.8,0.8) -- ++(-45:0.3) node[below right=-2] {\small $n$}
	(-0.8,0.8) -- ++(-35:0.3)
	(-0.8,0.8) -- ++(-55:0.3);


\node at (0,-1.1) {\small (4.1)};


\draw[xshift=2cm]
	(0.8,0.8) -- ++(45:-0.3);
\node[xshift=2cm] at (0,-1.1) {\small (4.2)};


\draw[xshift=4cm]
	(0.8,-0.8) -- ++(-45:-0.3);
\node[xshift=4cm] at (0,-1.1) {\small (4.3)};	


\draw[xshift=6cm]
	(0.8,0.8) -- ++(45:-0.3)
	(0.8,-0.8) -- ++(-45:-0.3);
\node[xshift=6cm] at (0,-1.1) {\small (4.4)};


\draw[xshift=8cm]
	(0.8,0.8) -- ++(45:-0.3)
	(-0.8,-0.8) -- ++(45:0.3);
\node[xshift=8cm] at (0,-1.1) {\small (4.5)};
	

\draw[xshift=10cm]
	(0.8,0.8) -- ++(45:-0.3)
	(-0.8,-0.8) -- ++(45:0.3)
	(0.8,-0.8) -- ++(-45:-0.3);
\node[xshift=10cm] at (0,-1.1) {\small (4.6)};

\end{scope}


\begin{scope}[yshift=-4.6cm]

\foreach \a in {0,...,4}
\foreach \b in {0,1}
{
\foreach \c in {0,...,4}
\draw[shift={(2.5*\a cm,-2.3*\b cm)},rotate=72*\c]
	(18:0.9) -- (90:0.9);

\draw[shift={(2.5*\a cm,-2.3*\b cm)}]
	(90:0.9) -- ++(-80:0.3)
	(90:0.9) -- ++(-90:0.3) node[below=-2] {\small $n$}
	(90:0.9) -- ++(-100:0.3);
}


\node at (0,-1) {\small (5.1)};


\draw[xshift=2.5cm]
	(18:0.9) -- ++(18:-0.3);
\node at (2.5,-1) {\small (5.2)};	


\draw[xshift=5cm]
	(-54:0.9) -- ++(-54:-0.3);
\node at (5,-1) {\small (5.3)};


\draw[xshift=7.5cm]
	(18:0.9) -- ++(18:-0.3)
	(-54:0.9) -- ++(-54:-0.3);
\node at (7.5,-1) {\small (5.4)};	


\draw[xshift=10cm]
	(18:0.9) -- ++(18:-0.3)
	(234:0.9) -- ++(234:-0.3);
\node at (10,-1) {\small (5.5)};	


\draw[yshift=-2.3cm]
	(18:0.9) -- ++(18:-0.3)
	(162:0.9) -- ++(162:-0.3);
\node at (0,-3.3) {\small (5.6)};


\draw[shift={(2.5cm,-2.3cm)}]
	(-54:0.9) -- ++(-54:-0.3)
	(234:0.9) -- ++(234:-0.3);
\node at (2.5,-3.3) {\small (5.7)};	


\draw[shift={(5cm,-2.3cm)}]
	(18:0.9) -- ++(18:-0.3)
	(-54:0.9) -- ++(-54:-0.3)
	(234:0.9) -- ++(234:-0.3);
\node at (5,-3.3) {\small (5.8)};	


\draw[shift={(7.5cm,-2.3cm)}]
	(18:0.9) -- ++(18:-0.3)
	(162:0.9) -- ++(162:-0.3)
	(-54:0.9) -- ++(-54:-0.3);
\node at (7.5,-3.3) {\small (5.9)};	


\draw[shift={(10cm,-2.3cm)}]
	(18:0.9) -- ++(18:-0.3)
	(162:0.9) -- ++(162:-0.3)
	(-54:0.9) -- ++(-54:-0.3)
	(234:0.9) -- ++(234:-0.3);
\node at (10,-3.3) {\small (5.10)};

\end{scope}

\end{tikzpicture}
\caption{All boundary configurations, $m\le 5$.}
\label{allcase}
\end{figure}
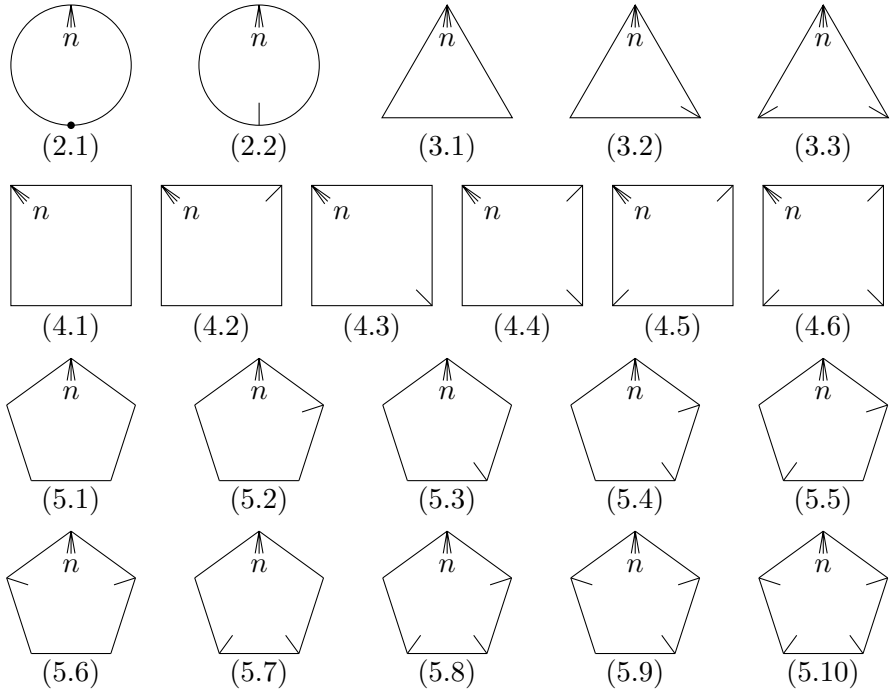

By Figure \ref{impossible}, the configurations (2.1), (3.1) and (4.1) are impossible. For the remaining 18 configurations, if $n$ is sufficiently large, then we may try to build the tiles one by one by making use of the degree $3$ assumption and Lemma \ref{sequence}. When the remaining region becomes one of the configurations in Figure \ref{allcase}, we say that the original configuration is reduced to the new configuration.

Figures \ref{3cycle} and \ref{5cycle} give all the major (so called \emph{generic}) reductions. There are also numerous \emph{non-generic} reductions and some special reductions. The upshot is that the reductions allow us to prove by inducting on the number of tiles. The lemma is the consequence of the fact that, unless we start with $(5.1)_{n=0}$ or $(5.10)_{n=1}$, the reductions never conclude with $(5.1)_{n=0}$ or $(5.10)_{n=1}$. Therefore in the subsequent proof, an argument ends (or a case is dismissed) once we arrive at a configuration (sometimes one configuration among several) in Figure \ref{allcase} that is neither $(5.1)_{n=0}$ nor $(5.10)_{n=1}$.

Figure \ref{3cycle} presents the generic reductions for the configurations (2.2), (3.2) and (3.3). The tiles are constructed with the help of Lemma \ref{sequence}, the degree $3$ assumption, and the additional assumption (hence the name ``generic'') that any newly created vertex is not identified with an existing one unless it is absolutely necessary. We also note that, whenever an edge from the high degree vertex (i.e., the vertex of degree $n+2$ on the boundary) is used, we only use the edge closest to the boundary. This means that the angles $\alpha$ in Figure \ref{3cycle} do not contain any more edges.

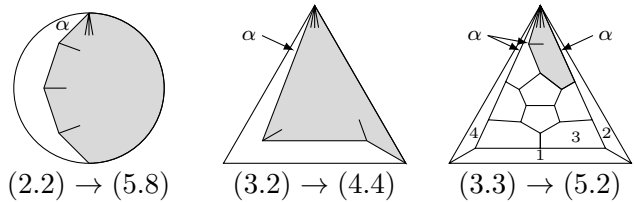
\begin{figure}[h]
\centering
\begin{tikzpicture}[>=latex,scale=1]

\filldraw[gray!30,draw=black]
	(0,1) -- (-0.4,0.6) -- (-0.6,0) -- (-0.4,-0.6) -- (0,-1) arc (-90:90:1);
\draw
	(0,0) circle (1)	
	(0,1) -- ++(-80:0.3)
	(0,1) -- ++(-90:0.3)
	(0,1) -- ++(-100:0.3)
	(-0.4,0.6) -- ++(-20:0.3)
	(-0.6,0) -- ++(0:0.3)
	(-0.4,-0.6) -- ++(20:0.3);
\node at (0,-1.3) {\small $(2.2)\to (5.8)$};
\node at (-0.35,0.83) {\scriptsize $\alpha$};

\begin{scope}[shift={(3cm,-0.3cm)}]
{
\filldraw[gray!30,draw=black]	
	(90:1.4) -- (210:0.8) -- (-30:0.8) -- (-30:1.4) -- cycle;
\draw	
	(0,1.4) -- ++(-80:0.3)
	(0,1.4) -- ++(-90:0.3)
	(0,1.4) -- ++(-100:0.3)
	(90:1.4) -- (210:1.4) -- (-30:1.4) -- cycle
	(210:0.8) -- (210:0.5)
	(-30:0.8) -- ++(110:0.3);
\node at (0,-1) {\small $(3.2)\to (4.4)$};
\draw[->]
	(-0.7,1) node[left=-2] {\scriptsize $\alpha$} -- (-0.27,0.8);
}
\end{scope}

\begin{scope}[shift={(6cm,-0.3cm)}]
{
\filldraw[gray!30,draw=black]	
	(90:1.4) -- (40:0.6) -- (45:0.4) -- (90:0.5) -- (100:0.9) -- cycle;
\draw	
	(0,1.4) -- ++(-80:0.3)
	(0,1.4) -- ++(-90:0.3)
	(0,1.4) -- ++(-100:0.3)
	(90:1.4) -- (210:1.4) -- node[above=-3] {\tiny 1} (-30:1.4) -- cycle
	(210:1.4) -- (210:1) -- (-30:1) -- (-30:1.4)
	(210:1) -- (190:0.7) -- (140:0.6) -- (90:1.4)
	(-30:1) -- (-10:0.7) -- (40:0.6) -- (90:1.4)
	(190:0.7) -- (210:0.3) -- (-90:0.3) -- (-90:0.5)
	(-10:0.7) -- node[below] {\tiny 3} (-30:0.3) -- (-90:0.3) -- (-90:0.5)
	(210:0.3) -- (160:0.2) -- (20:0.2) -- (-30:0.3)
	(160:0.2) -- (135:0.4) -- (140:0.6)
	(20:0.2) -- (45:0.4) -- (40:0.6)
	(45:0.4) -- (90:0.5) -- (135:0.4)
	(100:0.9) -- ++(0.2,0);
\node at (0,-1) {\small $(3.3)\to (5.2)$};
\node at (0.88,-0.3) {\tiny 2};
\node at (-0.88,-0.3) {\tiny 4};
\draw[->]
	(-0.7,1) node[left=-2] {\scriptsize $\alpha$} -- (-0.27,0.8);
\draw[->]
	(-0.7,1) -- (-0.16,0.9);
\draw[->]
	(0.7,1) node[right=-2] {\scriptsize $\alpha$} -- (0.27,0.8);
}
\end{scope}

\end{tikzpicture}
\caption{Generic reductions for two or three boundary edges.}
\label{3cycle}
\end{figure}

Figure \ref{otherred} shows many non-generic reductions for (3.3). Note that the non-generic reductions may conclude with more than one configurations. For example, the second reduction gives (2.1) and (4.4), and the third reduction gives (3.1) and $(5.1)_{n=0}$. For the purpose of proving the lemma, we only need one of these not to be $(5.1)_{n=0}$ or $(5.10)_{n=1}$. 

The reductions in Figure \ref{otherred} are created as follows. First by Lemma \ref{sequence}, we may create the pentagonal tile $1$ in Figure \ref{3cycle}. This creates three new vertices. Since the five vertices of the tile 1 are distinct (because of the simple closed curve assumption), if there is any identification of a new vertex with the existing one, it has to be identified with the high degree vertex. Up to symmetry, the first two pictures in Figure \ref{otherred} are all such identifications. Now we may assume the three new vertices are not identified with the existing ones and continue constructing the pentagonal tile $2$ in Figure \ref{3cycle}. The two new vertices cannot be identified with the high degree vertex because the five vertices of the tile 2 must be distinct. When the two new vertices are identified with the existing vertices of degree $3$, we get the third or the fourth picture. Then we may assume the two new vertices are ``really new'', and continue constructing the pentagonal tile $3$ in Figure \ref{3cycle}. The fifth and six pictures then indicate the cases that one of the vertices of the tile $3$ is identified with the high degree vertex.

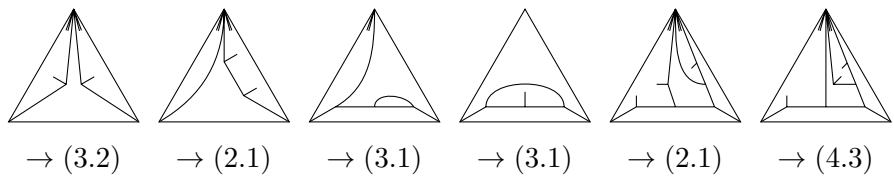
\begin{figure}[h]
\centering
\begin{tikzpicture}[>=latex,scale=1]

\foreach \x in {0,...,5}
\draw[xshift=2*\x cm]
	(90:1) -- (210:1) -- (-30:1) --cycle;

\draw
	(90:1) -- (180:0.1) -- (210:1)
	(90:1) -- (0:0.1) -- (-30:1)
	(180:0.1) -- ++(150:0.2)
	(0:0.1) -- ++(30:0.2)
	(90:1) -- ++(70:-0.3)
	(90:1) -- ++(75:-0.3)
	(90:1) -- ++(-70:0.3)
	(90:1) -- ++(-75:0.3);
\node at (0,-1) {\small $\to (3.2)$};

\draw[xshift=2cm]
	(90:1) to[out=-90,in=30] (210:1)
	(90:1) -- (90:0.3) -- (-30:0.3) -- (-30:1)
	(90:0.3) -- ++(30:0.2)
	(-30:0.3) -- ++(30:0.2)
	(90:1) -- ++(70:-0.3)
	(90:1) -- ++(75:-0.3)
	(90:1) -- ++(-70:0.3)
	(90:1) -- ++(-75:0.3);
\node[xshift=2cm] at (0,-1) {\small $\to (2.1)$};

\draw[xshift=4cm]
	(210:1) -- (210:0.6) -- (-30:0.6) -- (-30:1)
	(-30:0.6) to[out=120,in=90] (-90:0.3)
	(210:0.6) to[out=30,in=-90] (90:1)
	(90:1) -- ++(70:-0.3)
	(90:1) -- ++(75:-0.3);
\node[xshift=4cm] at (0,-1) {\small $\to (3.1)$};

\draw[xshift=6cm]
	(210:1) -- (210:0.6) -- (-30:0.6) -- (-30:1)
	(-30:0.6) to[out=90,in=90] (210:0.6)
	(-90:0.3) -- ++(90:0.2);
\node[xshift=6cm] at (0,-1) {\small $\to (3.1)$};

\draw[xshift=8cm]
	(210:1) -- (210:0.6) -- (-30:0.6) -- (-30:1)
	(-90:0.3) -- (-0.1,0) -- (90:1)
	(90:1) to[out=-90,in=180] (0:0.4)
	(90:1) -- (45:0.4) -- (0:0.4) -- (-30:0.6)
	(45:0.4) -- ++(45:-0.1)
	(210:0.6) -- ++(90:0.15)
	(-0.1,0) -- ++(-0.15,0)
	(90:1) -- ++(-80:0.3)
	(90:1) -- ++(-75:0.3)
	(90:1) -- ++(-110:0.3)
	(90:1) -- ++(-105:0.3);
\node[xshift=8cm] at (0,-1) {\small $\to (2.1)$};

\draw[xshift=10cm]
	(210:1) -- (210:0.6) -- (-30:0.6) -- (-30:1)
	(-90:0.3) -- (90:1)
	(90:1) -- (0.1,0) -- (0:0.4)
	(90:1) -- (45:0.4) -- (0:0.4) -- (-30:0.6)
	(45:0.4) -- ++(45:-0.1)
	(210:0.6) -- ++(90:0.15)
	(0.1,0) -- ++(45:0.15)
	(90:1) -- ++(-80:0.3)
	(90:1) -- ++(-75:0.3)
	(90:1) -- ++(-110:0.3)
	(90:1) -- ++(-105:0.3);
\node[xshift=10cm] at (0,-1) {\small $\to (4.3)$};

\end{tikzpicture}
\caption{Other possible reductions of $(3.3)$.}
\label{otherred}
\end{figure}

Analyzing each non-generic reduction at each step of the construction is rather tedious. An easy way to dismiss the non-generic reductions is to look at the distance between a new vertex and an existing one. For example, the gray region in Figure \ref{otherred2} is a newly constructed pentagonal tile, with two new vertices. The boundary of the remaining region of the triangle consists of $9$ edges. Therefore if any of the new vertex is identified with the existing one, one of the regions we get will have no more than four edges on the boundary, which is of the form $(2.*)$, $(3.*)$ or $(4.*)$. Since this is neither $(5.1)_{n=0}$ nor $(5.10)_{n=1}$, such a reduction fits into our proof.

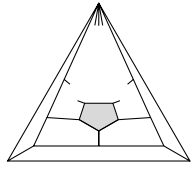
\begin{figure}[h]
\centering
\begin{tikzpicture}[>=latex,scale=1]

\filldraw[gray!30,draw=black]	
	(210:0.3) -- (160:0.2) -- (20:0.2) -- (-30:0.3) -- (-90:0.3) -- cycle;
\draw	
	(0,1.4) -- ++(-80:0.3)
	(0,1.4) -- ++(-90:0.3)
	(0,1.4) -- ++(-100:0.3)
	(90:1.4) -- (210:1.4) -- (-30:1.4) -- cycle
	(210:1.4) -- (210:1) -- (-30:1) -- (-30:1.4)
	(210:1) -- (190:0.7) -- (140:0.6) -- (90:1.4)
	(-30:1) -- (-10:0.7) -- (40:0.6) -- (90:1.4)
	(190:0.7) -- (210:0.3) -- (-90:0.3) -- (-90:0.5)
	(-10:0.7) -- (-30:0.3) -- (-90:0.3) -- (-90:0.5)
	(160:0.2) -- (160:0.3)
	(140:0.5) -- (140:0.6)
	(20:0.2) -- (20:0.3)
	(40:0.5) -- (40:0.6);

\end{tikzpicture}
\caption{Easy way to dismiss non-generic reductions.}
\label{otherred2}
\end{figure}

In general, if after constructing a new tile, the remaining region has no more than $10$ boundary edges, then the identification of the new and existing vertices will lead to reductions. With this observation, we may easily dismiss all the non-generic reductions along the way and get the generic reductions in Figure \ref{5cycle} for the $4$-gons and $5$-gons. The question mark means a configuration from Figure \ref{impossible}. We also note that the configuration (5.1) leads to a contradiction unless $n=0$.

\begin{figure}[h]
\centering
\begin{tikzpicture}[>=latex,scale=1]


\foreach \x in {0,1,2,3,4}
\draw[xshift=2.5*\x cm]
	(0,0) rectangle (2,2);


\filldraw[gray!30,draw=black]
	(0,2) -- (1,0.5) -- (2,2) -- cycle;
\draw	
	(0,2) -- ++(-45:0.3)
	(0,2) -- ++(-35:0.3)
	(1,0.5) -- (1,0.8);
\node at (1,-0.3) {\small $(4.2)\to (3.2)$};


\begin{scope}[xshift=2.5cm]
{
\filldraw[gray!30, draw=black]
	(0,2) -- (0.5,0.5) -- (1.5,0.5) -- (2,0) -- (2,2) -- cycle;
\draw	
	(0,2) -- ++(-45:0.3)
	(0,2) -- ++(-35:0.3)
	(0,2) -- ++(-55:0.3)
	(1.5,0.5) -- ++(0,0.3)
	(0.5,0.5) -- ++(45:0.3);
\node at (1,-0.3) {\small $(4.3)\to (5.4)$};
}
\end{scope}


\begin{scope}[xshift=5cm]
{
\filldraw[gray!30, draw=black]
	(0,2) -- (0.5,0.5) -- (1.5,0.5) -- (2,0) -- (2,2) -- cycle;
\draw	
	(0,2) -- ++(-45:0.3)
	(0,2) -- ++(-35:0.3)
	(0,2) -- ++(-55:0.3)
	(1.5,0.5) -- ++(0,0.3)
	(0.5,0.5) -- ++(45:0.3)
	(2,2) -- ++(45:-0.3);
\node at (1,-0.3) {\small $(4.4)\to (5.9)$};
}
\end{scope}


\begin{scope}[xshift=7.5cm]
{
\filldraw[gray!30,draw=black]
	(0,2) -- (0,0) -- (1.2,0.3) -- (1.7,0.8) -- (2,2) -- cycle;
\draw	
	(0,2) -- ++(-45:0.3)
	(0,2) -- ++(-35:0.3)
	(0,2) -- ++(-55:0.3)
	(1.2,0.3) -- ++(135:0.3)
	(1.7,0.8) -- ++(135:0.3);
\node at (1,-0.3) {\small $(4.5)\to (5.7)$};
}
\end{scope}


\begin{scope}[shift={(11cm,1cm)}]
{
\filldraw[gray!30,draw=black]
	(-1,1) -- (-0.7,0) -- (-0.2,0.2) -- (0,0.7) -- cycle;
\draw
	(-1,1) -- ++(-45:0.3)
	(-1,1) -- ++(-35:0.3)
	(-1,1) -- ++(-55:0.3)
	(-1,1) -- (-0.7,0) -- (-0.7,-0.7) -- (0.7,-0.7) -- (0.7,0.7) -- (0,0.7) -- cycle
	(-0.7,0) -- (-0.2,0.2) -- (0,0.7)
	(-0.7,-0.3) -- (-0.2,-0.3) -- (0,-0.5) -- (0,-0.7)
	(0.3,0.7) -- (0.3,0.2) -- (0.5,0) -- (0.7,0)
	(0,-0.5) -- (0.5,0)
	(-0.2,-0.3) -- (0,0) -- (0.3,0.2)
	(0,0) -- (-0.2,0.2)
	(-1,-1) -- (-0.7,-0.7)
	(1,1) -- (0.7,0.7)
	(1,-1) -- (0.7,-0.7);
\node at (0,-1.3) {\small $(4.6)\to (4.1)$};
}
\end{scope}


\begin{scope}[shift={(1cm,-1.7cm)}]
{

\foreach \x in {0,1,2,3,4}
\draw[xshift=2.5*\x cm]
	(18:1) -- (90:1) -- (162:1) -- (234:1) -- (-54:1) -- cycle;


\draw
	(0,1) -- ++(-90:0.3)
	(0,1) -- ++(-80:0.3)
	(0,1) -- ++(-100:0.3);
\node at (0,-1.2) {\small $(5.1)\to \times$};
\node at (0,0.2) {\small ? };
\node at (0,-0.2) {\small if $n\ge 1$};


\begin{scope}[xshift=2.5cm]
{
\filldraw[gray!30,draw=black]
	(0,1) -- (18:1) -- ++(220:0.6) to[out=220, in=250] (-0.2,0.4) -- cycle;
\draw
	(0,1) -- ++(-90:0.3)
	(0,1) -- ++(-80:0.3)
	(0,1) -- ++(-70:0.3);
\node at (0,-1.2) {\small $(5.2)\to (2.1)$};
}
\end{scope}	


\begin{scope}[xshift=5cm]
{
\filldraw[gray!30,draw=black]
	(0,1) -- (198:0.5) -- (-54:1) -- (18:1) -- cycle;
\draw
	(0,1) -- ++(-90:0.3)
	(0,1) -- ++(-80:0.3)
	(0,1) -- ++(-70:0.3)
	(198:0.5) -- (198:0.2);
\node at (0,-1.2) {\small $(5.3)\to (4.2)$};
}
\end{scope}	
	

\begin{scope}[xshift=7.5cm]
{
\filldraw[gray!30,draw=black]
	(0,1) -- (198:0.5) -- (-54:1) -- (18:1) -- cycle;
\draw
	(0,1) -- ++(-90:0.3)
	(0,1) -- ++(-80:0.3)
	(0,1) -- ++(-70:0.3)
	(18:1) -- ++(18:-0.3) 
	(198:0.5) -- (198:0.2);
\node at (0,-1.2) {\small $(5.4)\to (4.5)$};
}
\end{scope}	


\begin{scope}[xshift=10cm]
{
\filldraw[gray!30,draw=black]
	(0,1) -- (162:1) -- (0,-0.5) -- (18:1) -- cycle;
\draw
	(0,1) -- ++(-90:0.3)
	(0,1) -- ++(-80:0.3)
	(0,1) -- ++(-100:0.3)
	(0,-0.2) -- (0,-0.5);
\node at (0,-1.2) {\small $(5.6)\to (4.3)$};
}
\end{scope}

}
\end{scope}


\begin{scope}[shift={(0.8cm,-4.5cm)},scale=0.8]
{
	
\foreach \x in {0,1,2,3,4}
\draw[xshift=3.2*\x cm]
	(18:1.5) -- (90:1.5) -- (162:1.5) -- (234:1.5) -- (-54:1.5) -- cycle;
	

\fill[gray!30]
	(0,1.5) -- (162:0.8) -- (0,0) -- (18:0.8) -- (54:0.9) -- cycle;
\draw	
	(0,1.5) -- ++(-80:0.4)
	(0,1.5) -- ++(-90:0.4)
	(0,1.5) -- ++(-100:0.4)
	(234:1.5) -- (-90:1) -- (-18:1) -- (18:1.5)
	(0,1.5) -- (54:0.9) -- (18:0.8) -- (-18:1)
	(0,1.5) -- (162:0.8) -- (-90:1)
	(162:0.8) -- (0,0) -- (18:0.8)
	(0,0) -- (0,0.4)
	(54:0.9) -- ++(198:0.3);
\node at (0,-1.7) {\small $(5.5)\to (5.5)$};


\begin{scope}[xshift=3.2cm]
{
\filldraw[gray!30,draw=black]
	(0,1.5) -- (126:0.8) -- (162:0.6) -- (18:0.6) -- (54:0.8) -- cycle;
\draw	
	(0,1.5) -- ++(-80:0.4)
	(0,1.5) -- ++(-90:0.4)
	(0,1.5) -- ++(-100:0.4)
	(0,1.5) -- (18:1) -- (-54:1) -- (-54:1.5)
	(0,1.5) -- (162:1) -- (234:1) -- (234:1.5)
	(234:1) -- (0,-0.9) -- (-54:1)
	(0,-0.9) -- (0,-0.6)
	(162:1) -- (234:0.6) -- (0,-0.6)
	(18:1) -- (-54:0.6) -- (0,-0.6)
	(-54:0.6) -- (18:0.6) -- (162:0.6) -- (234:0.6)
	(126:0.8) -- ++(-18:0.3)
	(54:0.8) -- ++(18:-0.3);
\node at (0,-1.7) {\small $(5.7)\to (5.6)$};
}
\end{scope}


\begin{scope}[xshift=6.4cm]
{
\fill[gray!30]
	(0,1.5) -- (126:0.5) -- (72:0.9) -- cycle;
\draw	
	(0,1.5) -- ++(-80:0.4)
	(0,1.5) -- ++(-90:0.4)
	(0,1.5) -- ++(-100:0.4)
	(234:1.5) -- (234:1) 
	(-54:1) -- (-54:1.5)
	(0,1.5) -- (162:1) -- (234:1) -- (0,-0.9) -- (-54:1) -- (-18:0.9) -- (18:1) -- (54:0.9) -- (72:0.9) -- cycle
	(18:1) -- (18:1.5)
	(54:0.9) -- (54:0.5)
	(-18:0.9) -- (-18:0.5)
	(0,-0.9)  -- (0,-0.5)
	(162:1) -- (198:0.5)
	(54:0.5) -- (0,0) -- (198:0.5) -- (-90:0.5) -- (-18:0.5) -- cycle
	(0,0) -- (126:0.5) -- (0,1.5)
	(126:0.5) -- (72:0.9);
\node at (0,-1.7) {\small $(5.8)\to (3.1)$};
}
\end{scope}


\begin{scope}[xshift=9.6cm]
{
\fill[gray!30]
	(0,1.5) -- (162:1) -- (0,0) -- (126:0.5) -- (72:0.9) -- cycle;
\draw	
	(0,1.5) -- ++(-80:0.4)
	(0,1.5) -- ++(-90:0.4)
	(0,1.5) -- ++(-100:0.4)
	(-54:1) -- (-54:1.5)
	(162:1.5) -- (234:1) -- (-54:1) -- (-18:0.9) -- (18:1) -- (54:0.9) -- (72:0.9) -- (0,1.5)
	(18:1) -- (18:1.5)
	(54:0.9) -- (54:0.5)
	(-18:0.9) -- (-18:0.5)
	(234:1) -- (234:0.5)
	(0,1.5) -- (162:1) -- (234:0.5) -- (-18:0.5) -- (54:0.5)
	(162:1) -- (0,0) -- (54:0.5)
	(0,0) -- (126:0.5) -- (72:0.9)
	(126:0.5) -- ++(-0.2,0);
\node at (0,-1.7) {\small $(5.9)\to (5.3)$};
}
\end{scope}


\begin{scope}[xshift=12.8cm]
{
\draw	
	(0,1.5) -- ++(-80:0.4)
	(0,1.5) -- ++(-90:0.4)
	(0,1.5) -- ++(-100:0.4)
	(0,1.5) -- (108:0.9) -- (126:0.9) -- (126:0.5) -- (54:0.5) -- (54:0.9) -- (72:0.9) -- cycle
	(18:1.5) -- (18:1)
	(162:1.5) -- (162:1)
	(234:1.5) -- (234:1) 
	(-54:1.5) -- (-54:1)
	(0,1.5)-- (108:0.9) -- (126:0.9) -- (162:1) -- (198:0.9) -- (234:1) -- (0,-0.9) -- (-54:1) -- (-18:0.9) -- (18:1) -- (54:0.9) -- (72:0.9) -- cycle
	(54:0.9) -- (54:0.5)
	(126:0.9) -- (126:0.5)
	(-18:0.9) -- (-18:0.5)
	(-90:0.9) -- (-90:0.5)
	(198:0.9) -- (198:0.5)
	(54:0.5) -- (126:0.5) -- (198:0.5) -- (-90:0.5) -- (-18:0.5) -- cycle
	(72:0.9) -- (72:0.7)
	(108:0.9) -- (108:0.7);
\node at (0,-1.7) {\small $(5.10)\to \times$};
\node at (0,0.7) {\small ?};
}
\end{scope}

}
\end{scope}

\end{tikzpicture}
\caption{Generic reductions for four or five boundary edges.}
\label{5cycle}
\end{figure}
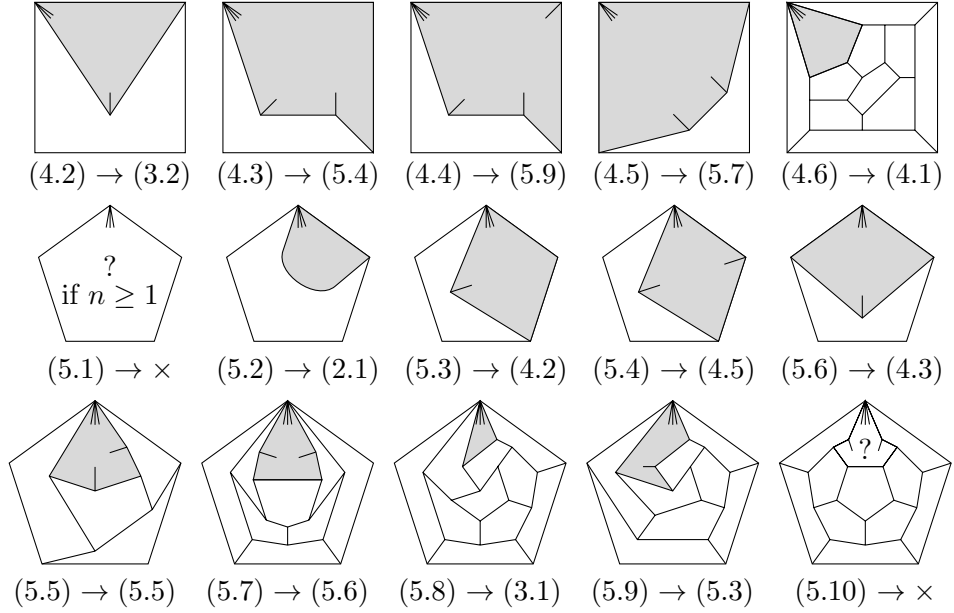

Note that $n=0$ means the configuration actually becomes another simpler configuration. For example, we have $(2.2)_{n=0}=(2.1)_{n=1}$, $(3.2)_{n=0}=(3.1)_{n=1}$, $(4.5)_{n=0}=(4.4)_{n=0}$, and $(5.5)_{n=0}=(5.3)_{n=1}$. Therefore we may always assume $n\ge 1$, with $(5.1)_{n=0}$ as the only exception. On the other hand, in the following cases, we need bigger $n$ to carry out the generic reductions. 
\begin{itemize}
\item $n\ge 2$ for (4.6), (5.5), (5.9), (5.10).
\item $n\ge 3$ for (3.3), (5.8).
\item $n\ge 4$ for (5.7).
\end{itemize} 
So it remains to reduce the special cases when the requirement on $n$ is not satisfied. These are listed in Figure \ref{specialcase}, with question marks meaning configurations from Figure \ref{impossible}.  The only one that does not lead to contradiction is $(5.10)_{n=1}$, which is the complement of one tile in the dodecahedron tiling.

\begin{figure}[h]
\centering
\begin{tikzpicture}[>=latex,scale=1]


\foreach \x in {0,1}
\draw[xshift=3*\x cm]
	(90:1.4) -- (210:1.4) -- (-30:1.4) -- cycle;


\draw
	(90:0.8) -- (150:0.5) -- (210:0.8) -- (-90:0.5) -- (-30:0.8) -- (30:0.5) -- cycle
	(-90:0.3) -- (30:0.3) -- (150:0.3) -- cycle
	(90:1.4) -- (90:0.8)
	(210:1.4) -- (210:0.8)
	(-30:1.4) -- (-30:0.8)
	(30:0.5) -- (30:0.3)
	(150:0.5) -- (150:0.3)
	(-90:0.5) -- (-90:0.3);
\node at (0,-1) {\small $(3.3)_{n=1}$};
\node at (0,-0.05) {\small ?};


\begin{scope}[xshift=3cm]
{
\draw
	(210:1.4) -- (210:1) -- (-30:1) -- (-30:1.4)
	(210:1) -- (190:0.7) -- (140:0.6) -- (90:1.4)
	(-30:1) -- (-10:0.7) -- (40:0.6) -- (90:1.4)
	(190:0.7) -- (210:0.3) -- (-90:0.3) -- (-90:0.5)
	(-10:0.7) -- (-30:0.3) -- (-90:0.3) -- (-90:0.5)
	(210:0.3) -- (160:0.2) -- (20:0.2) -- (-30:0.3)
	(160:0.2) -- (135:0.4) -- (140:0.6)
	(20:0.2) -- (45:0.4) -- (40:0.6)
	(45:0.4) -- (90:0.5) -- (135:0.4)
	(90:0.5) -- (90:0.7);
\node at (0,-1) {\small $(3.3)_{n=2}$};
\node at (90:0.9) {\small ?};
}
\end{scope}


\begin{scope}[shift={(6cm,0.3cm)}]
{
\foreach \x in {0,1,2,3}
\draw[rotate=90*\x]
	(-1,1) -- (1,1) -- (0.7,0.7) -- (-0.7,0.7)
	(0.4,0) -- (0,0.4) -- (0,0.7);
\node at (0,-1.3) {\small $(4.6)_{n=1}$};
\node at (0,0) {\small ?};
}
\end{scope}

\begin{scope}[shift={(9cm,0.3cm)}, scale=0.8]
{
\foreach \x in {0,1,...,4}
\draw[rotate=72*\x] 
	(90:1.5) -- (18:1.5) -- (18:1) 
	(-18:0.9) -- (18:1) -- (54:0.9) -- (54:0.5) -- (-18:0.5);
\node at (0,-1.6) {\small $(5.10)_{n=1}$};
}
\end{scope}


\begin{scope}[shift={(-1cm,-2.6cm)},scale=0.8]
{
	
\foreach \x in {0,1,...,4}
\foreach \y in {0,1,...,4}
\draw[xshift=3.3*\x cm, rotate=72*\y]
	(18:1.5) -- (90:1.5);
	

\draw
	(90:1.5) -- (90:1)
	(234:1.5) -- (234:1)
	(-54:1.5) -- (-54:1)
	(90:1) -- (234:1) -- (-54:1) -- cycle
	;
\node at (0,-1.6) {\small  $(5.5)_{n=1}$};
\node at (0,-2.2) {\small $(5.7)_{n=1}$};
\node at (0,0) {\small ?};


\begin{scope}[xshift=3.3cm]
{
\draw
	(0,1.5) -- (18:1) -- (-54:1) -- (-54:1.5)
	(0,1.5) -- (162:1) -- (234:1) -- (234:1.5)
	(234:1) -- (0,-0.9) -- (-54:1)
	(0,-0.9) -- (0,-0.6)
	(162:1) -- (234:0.6) -- (-90:0.6) -- (-54:0.6) -- (18:1)
	(234:0.6) to[out=90,in=180] (90:0.6) to[out=0,in=90] (-54:0.6);
\node at (0,-1.6) {\small  $(5.7)_{n=2}$};
\node at (0,0) {\small ?};
}
\end{scope}


\begin{scope}[xshift=6.6cm]
{
\draw
	(0,1.5) -- (18:1) -- (-54:1) -- (-54:1.5)
	(0,1.5) -- (162:1) -- (234:1) -- (234:1.5)
	(234:1) -- (0,-0.9) -- (-54:1)
	(0,-0.9) -- (0,-0.6)
	(162:1) -- (234:0.6) -- (-90:0.6) -- (-54:0.6) -- (18:1)
	(90:1.5) -- (90:1)
	(234:0.6) -- (162:0.6)
	(-54:0.6) -- (18:0.6)
	(90:1) -- (162:0.6) -- (18:0.6) -- cycle;
\node at (0,-1.6) {\small  $(5.7)_{n=3}$};
\node at (0,0.5) {\small ?};
}
\end{scope}


\begin{scope}[xshift=9.9cm]
{
\foreach \x in {0,1,2}
\draw[rotate=72*\x]
	(234:1.5) -- (234:1) -- (-90:0.9) -- (-54:1);
\draw
	(0,1.5)  -- (0,1) -- (234:1)
	(-90:0.9) -- (0,0) -- (54:0.9)
	(0,0) -- (-54:0.5) -- (-18:0.9)
	(-54:0.5) -- ++(72:0.4);
\node at (0,-1.6) {\small  $(5.8)_{n=1}$};
\node at (0,-2.2) {\small $(5.9)_{n=1}$};
\node at (0.6,0.3) {\small ?};
}
\end{scope}


\begin{scope}[xshift=13.2cm]
{
\draw
	(234:1.5) -- (234:1) 
	(-54:1) -- (-54:1.5)
	(0,1.5) -- (162:1) -- (234:1) -- (0,-0.9) -- (-54:1) -- (-18:0.9) -- (18:1) -- (54:0.9) -- (72:0.9) -- cycle
	(18:1) -- (18:1.5)
	(54:0.9) -- (54:0.5)
	(-18:0.9) -- (-18:0.5)
	(0,-0.9)  -- (0,-0.5)
	(162:1) -- (198:0.5)
	(54:0.5) -- (0,0) -- (198:0.5) -- (-90:0.5) -- (-18:0.5) -- cycle
	(0,0) -- (126:0.5) -- ++(-0.2,0.2)
	(126:0.5) -- (72:0.9);
\node at (0,-1.6) {\small $(5.8)_{n=2}$};
\node at (-0.2,0.8) {\small ?};
}
\end{scope}

}
\end{scope}

\end{tikzpicture}
\caption{Special reductions.}
\label{specialcase}
\end{figure}
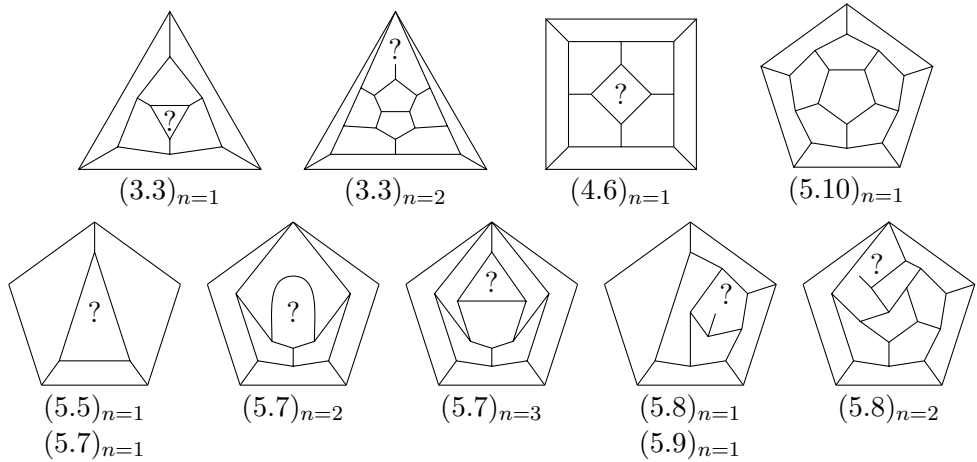

This completes the proof for $m\le 5$. Now we turn to $6$-gons and $7$-gons. There are many possible configurations to consider. However, observe that if part of the boundary is as illustrated by the thick lines in Figure \ref{easyred}, then we can always produce the reduction. Moreover, we allow a vertex at the fringe of the thick lines to be the high degree vertex (see lower left corner of the first picture). 

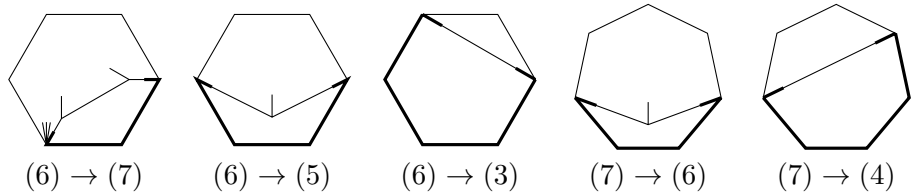
\begin{figure}[h]
\centering
\begin{tikzpicture}[>=latex,scale=1]


\foreach \x in {0,1,2}
\draw[xshift=2.5*\x cm]
	(0:1) -- (60:1) -- (120:1) -- (180:1) -- (240:1) -- (300:1) -- cycle;

\draw
	(240:1) -- ++(90:0.3)
	(240:1) -- ++(80:0.3)
	(240:1) -- ++(100:0.3)
	(0:1) -- (0:0.6) -- (240:0.6) -- (240:1)
	(240:0.6) -- ++(0,0.3)
	(0:0.6) -- ++(150:0.3);
\draw[very thick]
	(240:0.8) -- (240:1) -- (300:1) -- (0:1) -- (0:0.8);
\node at (0,-1.3) {\small $(6)\to (7)$};

\begin{scope}[xshift=2.5cm]
{
\draw
	(0:1) -- (-90:0.5) -- (180:1)
	(-90:0.5) -- (-90:0.2);
\draw[very thick]
	(-0.8,-0.1) -- (180:1) -- (240:1) -- (300:1) -- (0:1) -- (0.8,-0.1);
\node at (0,-1.3) {\small $(6)\to (5)$};
}
\end{scope}

\begin{scope}[xshift=5cm]
{
\draw
	(0:1) -- (120:1);
\draw[very thick]
	(120:1) -- (180:1) -- (240:1) -- (300:1) -- (0:1) 
	(0:1) -- ++(-30:-0.3)
	(120:1) -- ++(-30:0.3);
\node at (0,-1.3) {\small $(6)\to (3)$};
}
\end{scope}


\foreach \x in {0,1}
\draw[xshift=7.5cm+2.5*\x cm]
	(-14:1) -- (38:1) -- (90:1) -- (142:1) -- (194:1) -- (246:1) -- (-66:1) --  cycle;

\begin{scope}[xshift=7.5cm]
{
\draw
	(194:1) -- (0,-0.6) -- (-14:1)
	(0,-0.6) -- (0,-0.3);
\draw[very thick]
	(194:1) -- (246:1) -- (-66:1) -- (-14:1)
	(-14:1) -- ++(200:0.3)
	(194:1) -- ++(-20:0.3);
\node at (0,-1.3) {\small $(7)\to (6)$};
}
\end{scope}

\begin{scope}[xshift=10cm]
{
\draw
	(194:1) -- (38:1);
\draw[very thick]
	(194:1) -- (246:1) -- (-66:1) -- (-14:1) -- (38:1)
	(38:1) -- ++(206:0.3)
	(194:1) -- ++(26:0.3);
\node at (0,-1.3) {\small $(7)\to (4)$};
}
\end{scope}

\end{tikzpicture}
\caption{Easy reductions for six or seven boundary edges.}
\label{easyred}
\end{figure}

So we only need to study those configurations not in Figure \ref{easyred}. It is also easy to see that the configurations with one high degree vertex and one or two vertices of degree $3$ on the boundary either lead to contradictions or can be easily reduced. All the remaining non-trivial ones are illustrated in Figure \ref{7cycle}.

\begin{figure}[h]
\centering
\begin{tikzpicture}[>=latex,scale=1]



\foreach \x in {0,1,...,5}
\draw[rotate=60*\x]
	(0:1) -- (60:1);

\foreach \x in {0,1,...,4}
\draw[rotate=60*\x]
	(210:0.51) -- (210:0.3)
	(180:1) -- (180:0.6);

\foreach \x in {0,1,2,3}
\draw[rotate=60*\x]
	(180:0.6) -- (240:0.6)
	(210:0.3) -- (-90:0.3);

\draw
	(180:0.6) -- (90:0.51) -- (60:0.6)
	(210:0.3) -- (90:0.3)
	(120:1) -- ++(-60:0.3)
	(120:1) -- ++(-50:0.3)
	(120:1) -- ++(-70:0.3);
\node at (0,-1.3) {\small $(6)\to (5)$};


\foreach \x in {1,2,3}
\draw[xshift=2.5*\x cm]
	(-14:1) -- (38:1) -- (90:1) -- (142:1) -- (194:1) -- (246:1) -- (-66:1) --  cycle
	(90:1) -- ++(-90:0.3)
	(90:1) -- ++(-80:0.3)
	(90:1) -- ++(-100:0.3);

\foreach \x in {1,2,3}
{
\draw[xshift=2.5*\x cm]
	(90:1) -- (142:0.6) -- (194:0.6) -- (-66:0.6) -- (-66:1)
	(142:0.6) -- ++(-20:0.3)
	(194:1) -- (194:0.6)
	(-66:0.6) -- ++(50:0.3);
\node[xshift=2.5*\x cm] at (0,-1.3) {\small $(7)\to (7)$};
}

\draw[xshift=2.5cm]
	(38:1) -- (38:0.7);

\draw[xshift=5cm]
	(-14:1) -- (-14:0.7);

\draw[xshift=7.5cm]
	(38:1) -- (38:0.7)
	(-14:1) -- (-14:0.7);


\begin{scope}[shift={(-1.2cm,-2.7cm)}]
{
\foreach \x in {0,1,2,3,4}
\draw[xshift=2.5*\x cm]
	(-14:1) -- (38:1) -- (90:1) -- (142:1) -- (194:1) -- (246:1) -- (-66:1) --  cycle
	(90:1) -- ++(-90:0.3)
	(90:1) -- ++(-80:0.3)
	(90:1) -- ++(-100:0.3);

\draw
	(90:1) -- (142:0.6) -- (194:0.6) -- (246:0.6) -- (-14:0.6) -- (38:0.6) -- (0.1,0.3) -- (142:0.6)
	(194:1) -- (194:0.6)
	(246:1) -- (246:0.6)
	(-14:1) -- (-14:0.6)
	(38:1) -- (38:0.6) 
	(12:0.54) -- (220:0.54)
	(-66:0.14) -- (114:0.35)
	(0.1,0.3) -- ++(0,0.2);
\node at (0,-1.3) {\small $(7)\to (6)$};

\draw[xshift=2.5cm]
	(142:0.6) -- (246:0.6) -- (-14:0.6)
	(142:1) -- (142:0.6) -- ++(40:0.3)
	(246:1) -- (246:0.3)
	(-14:1) -- (-14:0.6) -- ++(90:0.3)
	(38:1) -- (38:0.7);
\node[xshift=2.5cm] at (0,-1.3) {\small $(7)\to (7)$};

\draw[xshift=5cm]
	(90:1) -- (142:0.6) -- (194:0.6) -- (246:0.6) -- (-66:0.6) -- (38:0.6)
	(142:0.6) -- (142:0.3)
	(194:1) --  (194:0.6)
	(246:1) -- (246:0.6)
	(-66:1) -- (-66:0.6)
	(38:1) -- (38:0.3)
	(142:0.3) -- (220:0.3) -- (-90:0.3) -- (38:0.3) -- (0,0.3) -- cycle
	(220:0.3) -- (220:0.55)
	(-90:0.3) -- (-90:0.55)
	(0,0.3) -- (0,0.5);
\node[xshift=5cm] at (0,-1.3) {\small $(7)\to (7)$};

\draw[xshift=7.5cm]
	(90:1) -- (142:0.6) -- (194:0.6) -- (246:0.6) -- (-66:0.6) -- (-14:0.6) -- (38:0.6) -- cycle
	(142:0.6) -- (142:0.3)
	(38:0.6) -- (38:0.3)
	(194:1) --  (194:0.6)
	(246:1) -- (246:0.6)
	(-66:1) -- (-66:0.6)
	(-14:1) -- (-14:0.6)
	(142:0.3) -- (220:0.3) -- (-90:0.3) -- (-40:0.3) -- (38:0.3) -- cycle
	(-40:0.3) -- (-40:0.55)
	(220:0.3) -- (220:0.55)
	(-90:0.3) -- (-90:0.55);
\node[xshift=7.5cm] at (0,-1.3) {\small $(7)\to (5)$};

\draw[xshift=10cm]
	(142:0.6) -- (246:0.6) -- (-66:0.6) -- (38:0.6)
	(142:1) -- (142:0.3) -- (-90:0.3) -- (38:0.3) -- (38:1) 
	(246:1) -- (246:0.6)
	(-66:1) -- (-66:0.6)
	(-90:0.3) -- (-90:0.55)
	(90:1) -- (55:0.7) -- (38:0.3)
	(90:1) -- (125:0.7) -- (142:0.3)
	(55:0.7) -- ++(-0.2,0)
	(125:0.7) -- ++(0.2,0);
\node[xshift=10cm] at (0,-1.3) {\small $(7)\to (6)$};

}
\end{scope}


\begin{scope}[yshift=-5.4cm]
{

\draw[xscale=-1]
	(-14:1) -- (38:1) -- (90:1) -- (142:1) -- (194:1) -- (246:1) -- (-66:1) --  cycle
	(142:1) -- (142:0.6) -- (194:0.6) -- (246:0.6) -- (-14:0.6) -- (38:0.6)
	(194:1) -- (194:0.6)
	(246:1) -- (246:0.6)
	(-14:1) -- (-14:0.6)
	(38:1) -- (38:0.6) 
	(12:0.54) -- (220:0.54)
	(168:0.54) -- (0,0) -- (64:0.54)
	(38:0.6) -- (90:0.6) -- (142:0.6)
	(-66:0.14) -- (0,0)
	(90:1) -- ++(-90:0.2)
	(90:1) -- ++(-80:0.2)
	(90:1) -- ++(-100:0.2)
	(90:0.6) -- ++(90:0.1);
\node at (0,-1.3) {\small $(7)\to (7)$};

\foreach \x / \c in {1/0,2/52,3/0}
{
\draw[xshift=2.5*\x cm,rotate=\c]
	(-14:1) -- (38:1) -- (90:1) -- (142:1) -- (194:1) -- (246:1) -- (-66:1) --  cycle
	(194:0.6) -- (246:0.6) -- (-66:0.6) -- (-14:0.6) -- (38:0.6)
	(194:1) --  (194:0.6)
	(246:1) -- (246:0.6)
	(-66:1) -- (-66:0.6)
	(-14:1) -- (-14:0.6)
	(38:1) -- (38:0.6)
	(220:0.3) -- (220:0.55)
	(-90:0.3) -- (-90:0.55)
	(-40:0.3) -- (-40:0.55)
	(12:0.3) -- (12:0.55)
	(220:0.3) -- (-90:0.3) -- (-40:0.3) -- (12:0.3) -- (116:0.1) -- cycle
	(194:0.6) -- (116:0.3) -- (38:0.6)
	(116:0.1) -- (116:0.3);
\node[xshift=2.5*\x cm] at (0,-1.3) {\small $(7)\to (7)$};
}

\draw[xshift=2.5cm] 
	(90:1) -- ++(-90:0.3)
	(90:1) -- ++(-80:0.3)
	(90:1) -- ++(-100:0.3);

\draw[xshift=5cm] 
	(90:1) -- ++(-100:0.3)
	(90:1) -- ++(-110:0.3)
	(90:1) -- ++(-120:0.3)
	(194:1) -- (194:0.7);

\draw[xshift=7.5cm] 
	(90:1) -- ++(-90:0.3)
	(90:1) -- ++(-80:0.3)
	(90:1) -- ++(-100:0.3)
	(142:1) -- (142:0.7);

}
\end{scope}

\end{tikzpicture}
\caption{Not so easy reductions for six or seven boundary edges.}
\label{7cycle}
\end{figure}
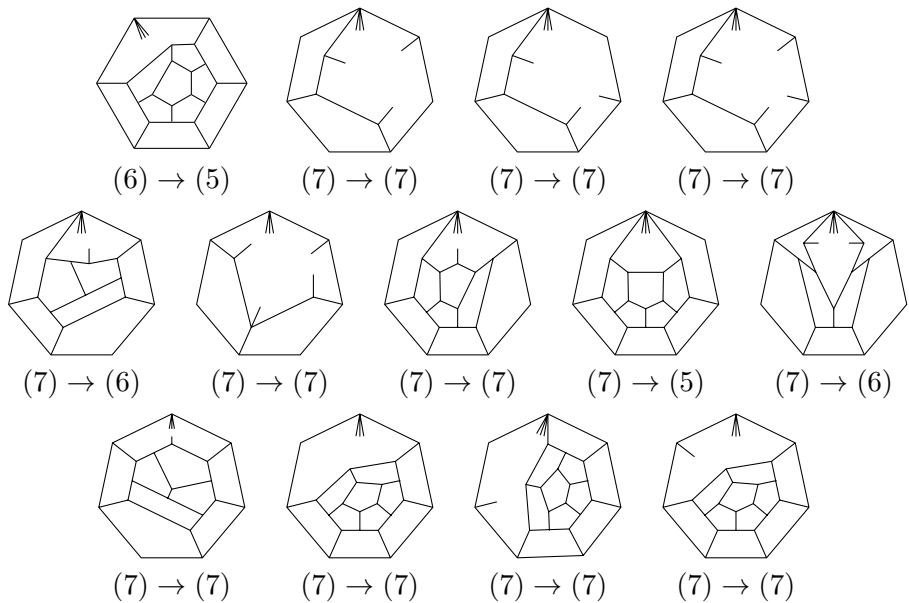

We need to be concerned about the non-generic reductions. Since we have included $6$-gons and $7$-gons in our study, a case can be dismissed if there are no more than $14$ boundary edges. The observation is sufficient for dismissing all the non-generic reductions.

We also need to consider the special cases in which there are not enough edges at the high degree vertex for carrying out the reductions. In this regard, we only need to worry about the case $n=1$ for the fourth and the fifth in the second row in Figure \ref{7cycle}. However, for $n=1$, they are respectively the second and the first in the second row. So the special cases have also been reduced.

Finally, we note that the reductions for the $6$-gons and $7$-gons never conclude with $(5.1)_{n=0}$ or $(5.10)_{n=1}$. Therefore all configurations eventually lead to contradictions.
\end{proof}

\section{Distance Between Vertices of Degree $>3$}

\begin{theorem}\label{distance5}
Suppose a combinatorial pentagonal tiling of the sphere has a vertex $\nu$ of degree $>3$, such that all vertices within distance $4$ of $\nu$ have degree $3$. Then the tiling is the earth map tiling in Figure \ref{map5}.
\end{theorem}

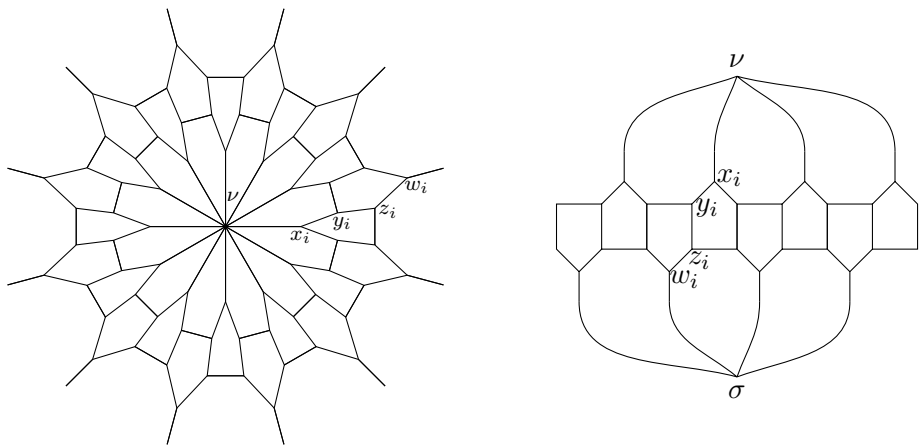
\begin{figure}[htp]
\centering
\begin{tikzpicture}[>=latex,scale=1]

\foreach \x in {0,1,...,11}
\foreach \y in {1,-1}
\draw[rotate=15+30*\x, yscale=\y]
		(0,0) -- (15:1)  -- (8:1.5) -- (8:2) -- (0:2.5) -- (0:3)
		(8:1.5) -- (-8:1.5)
		(8:2) -- (22:2);

\node at (0.1,0.4) {\scriptsize $\nu$};
\node at (1,-0.14) {\scriptsize $x_i$};
\node at (1.55,0.04) {\scriptsize $y_i$};
\node at (2.16,0.2) {\scriptsize $z_i$};
\node at (2.54,0.53) {\scriptsize $w_i$};

\begin{scope}[xshift=5cm]
{
\foreach \x in {0,1,2,3}
\draw[xshift=1.2*\x cm]
	(0.3,1) -- (0.3,0.6) -- (0,0.3) -- (0,-0.3) -- (-0.3,-0.6) -- (-0.3,-1) 
	(0,0.3) -- (-0.6,0.3) -- (-0.6,-0.3) -- (-0.3,-0.6)
	(0,-0.3) -- (0.6,-0.3) -- (0.6,0.3) -- (0.3,0.6);

\foreach \x in {1,-1}
\draw[xshift=1.8cm,scale=\x]
	(0,2) to[out=200,in=90] (-1.5,1)
	(0,2) to[out=250,in=90] (-0.3,1)
	(0,2) to[out=-40,in=90] (0.9,1)
	(0,2) to[out=-15,in=90] (2.1,1);

\node at (1.8,-2.2) {\small $\sigma$};
\node at (1.8,2.2) {\small $\nu$};
\node at (1.7,0.65) {\small $x_i$};
\node at (1.4,0.25) {\small $y_i$};
\node at (1.3,-0.45) {\small $z_i$};
\node at (1.1,-0.7) {\small $w_i$};
}
\end{scope}

\end{tikzpicture}
\caption{Earth map tiling, distance between poles $=5$.}
\label{map5}
\end{figure}

The tiling must have exactly two vertices $\nu$ and $\sigma$ (as the \emph{north} and \emph{south} poles) of degree $>3$, and the degrees at the two vertices must be the same. The first picture in Figure \ref{map5} is the view of the tiling from the north pole $\nu$. All the outward rays converge at the south pole $\sigma$. The second picture is our usual way of drawing maps, from the viewpoint of the equator.

An immediate consequence is that, if a combinatorial pentagonal tiling is not the earth map tiling in Figure \ref{map5}, then every vertex of degree $>3$ must have another vertex of degree $>3$ within distance $4$.

\begin{proof}
Starting with a vertex $\nu$ of degree $k>3$, we construct the pentagonal tiles layer by layer. We label the edges emanating from $\nu$ by $k$ distinct indices $i$ and denote the other end of the $i$-th edge by $x_i$. Since $x_*$ are of distance $1$ from $\nu$, they have degree $3$. If two $x_*$ are identified, say $x_i=x_j$, then we get a $2$-gon with $\nu$ and $x_i$ as the two vertices. Since $x_i$ has degree $3$, there are two possibilities for the $2$-gon, given in Figure \ref{xx}. The first appears in Figure \ref{impossible}, and the outside of the second is the first. This proves that all $x_*$ are distinct.

\begin{figure}[htp]
\centering
\begin{tikzpicture}[>=latex,scale=1]

\foreach \x in {0,1}
\draw[xshift=3*\x cm]
	(0,0) circle (0.8)
	(0,0.8) node[above right=-2] {\small $\nu$} -- ++(-90:0.3)
	(0,0.8) -- ++(-80:0.3)
	(0,0.8) -- ++(-100:0.3)
	(0,0.8) -- ++(90:0.3)
	(0,0.8) -- ++(80:0.3)
	(0,0.8) -- ++(100:0.3);

\draw
	(0,-0.8) node[above] {\scriptsize $x_i=x_j$} -- ++(-90:0.3);
	
\draw[xshift=3cm]
	(0,-0.8) node[below=-2] {\scriptsize $x_i=x_j$} -- ++(90:0.3);	

\fill 
	(0,-0.8) circle (0.05)
	(3,-0.8) circle (0.05);

\end{tikzpicture}
\caption{Two cases of $x_i=x_j$.}
\label{xx}
\end{figure}
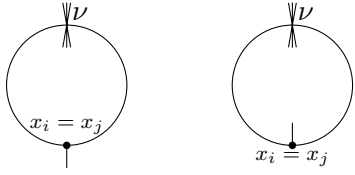

Each degree $3$ vertex $x_i$ is connected to two vertices $y_*$ in addition to $\nu$. Since the distance between $y_*$ and $\nu$ are $\le 2$, $y_*$ have degree $3$. If $x_i$ and $y_j$ are identified, then there are three possibilities, given in the first column of Figure \ref{x=y=z}. If two $y_*$ are identified, say $y_i=y_j$, then the identified vertex is a dot in the middle of the $x_ix_j$-edges in the first column, and we get the second column. Moreover, the third edge at the degree $3$ vertex $y_i$ may point either upward or downward, so that we actually have six possibilities in total.

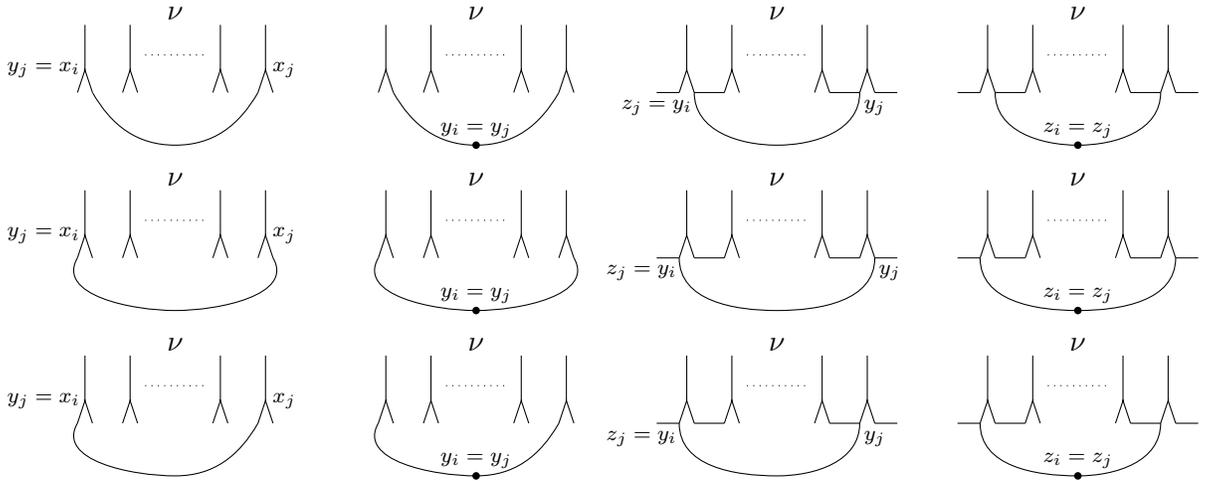
\begin{figure}[htp]
\centering
\begin{tikzpicture}[>=latex,scale=1]

\foreach \a in {0,1,2,3}
\foreach \b in {0,1,2}
{
\begin{scope}[shift={(4*\a cm,-2.2*\b cm)}]

\foreach \x in {-2,-1,1,2}
\draw[xshift=0.6*\x cm]
	(0,0.6) -- (0,0)
	(-0.1,-0.3)  -- (0,0) -- (0.1,-0.3);
	
\draw[dotted]
	(-0.4,0.2) -- (0.4,0.2);
\node at (0,0.75) {\small $\nu$};

\end{scope}
}


\foreach \a in {0,1}
{
\draw[xshift=4*\a cm]
	(-1.1,-0.3) to[out=-60,in=180] (0,-1) to[out=0,in=240] (1.1,-0.3);
\draw[shift={(4*\a cm,-2.2cm)}]
	(-1.3,-0.3) to[out=240,in=180] (0,-1) to[out=0,in=-60] (1.3,-0.3);
\draw[shift={(4*\a cm,-4.4cm)}]
	(-1.3,-0.3) to[out=240,in=180] (0,-1) to[out=0,in=240] (1.1,-0.3);
}


\foreach \b in {0,1,2}
{
\node at (-1.75,-2.2*\b cm) {\scriptsize $y_j=x_i$};
\node at (1.45,-2.2*\b cm) {\scriptsize $x_j$};
}


\foreach \b in {0,1,2}
{
\fill (4,-1cm -2.2*\b cm) circle (0.05);
\node at (4,-0.8 cm -2.2*\b cm) {\scriptsize $y_i=y_j$};
}


\foreach \a in {2,3}
\foreach \b in {0,1,2}
\draw[shift={(4*\a cm,-2.2*\b cm)}]
	(-1.1,-0.3) -- (-0.7,-0.3)
	(1.1,-0.3) -- (0.7,-0.3)
	(-1.3,-0.3) -- (-1.6,-0.3)
	(1.3,-0.3) -- (1.6,-0.3);
	
\foreach \a in {2,3}
{
\draw[xshift=4*\a cm] 
	(-1.1,-0.3) to[out=-90,in=180] (0,-1) to[out=0,in=-90] (1.1,-0.3);
\draw[shift={(4*\a cm,-2.2cm)}]
	(-1.3,-0.3) to[out=-90,in=180] (0,-1) to[out=0,in=-90] (1.3,-0.3);
\draw[shift={(4*\a cm,-4.4cm)}]
	(-1.3,-0.3) to[out=-90,in=180] (0,-1) to[out=0,in=-90] (1.1,-0.3);
}


\node at (6.4,-0.5) {\scriptsize $z_j=y_i$};
\node at (9.3,-0.5) {\scriptsize $y_j$};

\node at (6.2,-2.7) {\scriptsize $z_j=y_i$};
\node at (9.5,-2.7) {\scriptsize $y_j$};

\node at (6.2,-4.9) {\scriptsize $z_j=y_i$};
\node at (9.3,-4.9) {\scriptsize $y_j$};


\foreach \b in {0,1,2}
{
\fill (12,-1cm -2.2*\b cm) circle (0.05);
\node at (12,-0.8 cm -2.2*\b cm) {\scriptsize $z_i=z_j$};
}

\end{tikzpicture}
\caption{Identifications among $x_*,y_*,z_*$.}
\label{x=y=z}
\end{figure}

If all $x_*,y_*$ are distinct, then we may apply Lemma \ref{sequence} to construct all the tiles at the vertex $\nu$, as given by Figure \ref{map5}. Each tile has $\nu$, two $x_*$ and two $y_*$ as the five vertices. Now each degree $3$ vertex $y_*$ is connected to one $x_*$, another $y_*$, and a new vertex $z_*$. Since the distance between $z_*$ and $\nu$ are $\le 3$, $z_*$ have degree $3$. If $y_i$ and $z_j$ are identified, then we get three possibilities in the third column of Figure \ref{x=y=z}. If two $z_*$ are identified, say $z_i=z_j$, then the identified vertex is a dot in the middle of the $y_iy_j$-edges in the third column, and we get the fourth column. Moreover, considering the direction of the third edge at the dot, we have six possibilities in total.

We claim that all the configurations in Figure \ref{x=y=z} are impossible, so that all $x_*,y_*,z_*$ are distinct. We redraw the configurations as Figure \ref{xyz}. For example, the configurations in the first row of Figure \ref{x=y=z} are respectively labeled $(3.1)$, $(4.1)$, $(4.2)$, $(5.1)$, $(6.1)$, $(6.2)$. The vertex $\nu$ is indicated by $n$ interior edges at the high degree boundary vertex. 

\begin{figure}[h]
\centering
\begin{tikzpicture}[>=latex,scale=1]



\foreach \b in {0,1,2}
\draw[yshift=-0.3cm-2.6*\b cm]
	(-30:1) -- (90:1) -- (210:1) -- cycle
	(90:1) -- ++(-80:0.3)
	(90:1) -- ++(-90:0.3) node[below=-2] {\small $n$}
	(90:1) -- ++(-100:0.3);


\draw[yshift=-0.3cm]
	(-30:1) -- ++(-30:0.3)
	(210:1) -- ++(210:0.3);
\node at (0,-1.4) {\small $(3.1)$};


\draw[yshift=-2.9cm]
	(-30:1) -- ++(-30:-0.3)
	(210:1) -- ++(210:-0.3);
\node at (0,-4) {\small $(3.2)$};


\draw[yshift=-5.5cm]
	(-30:1) -- ++(-30:0.3)
	(210:1) -- ++(210:-0.3);
\node at (0,-6.6) {\small $(3.3)$};


\foreach \a in {1,2}
\foreach \b in {0,1,2}
\draw[shift={(2.5*\a cm,-2.6*\b cm)}]
	(-0.8,-0.8) rectangle (0.8,0.8)
	(-0.8,0.8) -- ++(-45:0.3) node[below right=-2] {\small $n$}
	(-0.8,0.8) -- ++(-35:0.3)
	(-0.8,0.8) -- ++(-55:0.3);


\draw[xshift=2.5cm]
	(0.8,0.8) -- ++(45:0.3)
	(-0.8,-0.8) -- ++(45:-0.3)
	(0.8,-0.8) -- ++(-45:0.3);
\node at (2.5,-1.4) {\small $(4.1)$};


\draw[xshift=5cm]
	(0.8,0.8) -- ++(45:0.3)
	(-0.8,-0.8) -- ++(45:-0.3)
	(0.8,-0.8) -- ++(-45:-0.3);
\node at (5,-1.4) {\small $(4.2)$};


\draw[shift={(2.5cm,-2.6cm)}]
	(0.8,0.8) -- ++(45:-0.3)
	(-0.8,-0.8) -- ++(45:0.3)
	(0.8,-0.8) -- ++(-45:0.3);
\node at (2.5,-4) {\small $(4.3)$};


\draw[shift={(5cm,-2.6cm)}]
	(0.8,0.8) -- ++(45:-0.3)
	(-0.8,-0.8) -- ++(45:0.3)
	(0.8,-0.8) -- ++(-45:-0.3);
\node at (5,-4) {\small $(4.4)$};


\draw[shift={(2.5cm,-5.2cm)}]
	(0.8,0.8) -- ++(45:0.3)
	(-0.8,-0.8) -- ++(45:0.3)
	(0.8,-0.8) -- ++(-45:0.3);
\node at (2.5,-6.6) {\small $(4.5)$};


\draw[shift={(5cm,-5.2cm)}]
	(0.8,0.8) -- ++(45:0.3)
	(-0.8,-0.8) -- ++(45:0.3)
	(0.8,-0.8) -- ++(-45:-0.3);
\node at (5,-6.6) {\small $(4.6)$};


\foreach \b in {0,1,2}
{
\foreach \c in {0,...,4}
\draw[shift={(7.5 cm,-0.1 cm -2.6*\b cm)},rotate=72*\c]
	(18:0.9) -- (90:0.9);

\draw[shift={(7.5 cm,-2.6*\b cm)}]
	(90:0.8) -- ++(-80:0.3)
	(90:0.8) -- ++(-90:0.3) node[below=-2] {\small $n$}
	(90:0.8) -- ++(-100:0.3);
}


\draw[shift={(7.5 cm,-0.1 cm)}]
	(18:0.9) -- ++(18:0.3)
	(162:0.9) -- ++(162:0.3)
	(234:0.9) -- ++(234:-0.3)
	(-54:0.9) -- ++(-54:-0.3);
\node at (7.5,-1.4) {\small $(5.1)$};


\draw[shift={(7.5 cm,-2.7cm)}]
	(18:0.9) -- ++(18:-0.3)
	(162:0.9) -- ++(162:-0.3)
	(234:0.9) -- ++(234:0.3)
	(-54:0.9) -- ++(-54:0.3);
\node at (7.5,-4) {\small $(5.2)$};


\draw[shift={(7.5 cm,-5.3cm)}]
	(18:0.9) -- ++(18:-0.3)
	(162:0.9) -- ++(162:0.3)
	(234:0.9) -- ++(234:-0.3)
	(-54:0.9) -- ++(-54:0.3);
\node at (7.5,-6.6) {\small $(5.3)$};


\foreach \a in {4,5}
\foreach \b in {0,1,2}
{
\foreach \c in {0,...,5}
\draw[shift={(2.5*\a cm,-2.6*\b cm)},rotate=60*\c]
	(-30:0.9) -- (30:0.9);

\draw[shift={(2.5*\a cm,-2.6*\b cm)}]
	(90:0.9) -- ++(-80:0.3)
	(90:0.9) -- ++(-90:0.3) node[below=-2] {\small $n$}
	(90:0.9) -- ++(-100:0.3);
}


\draw[xshift=10cm]
	(30:0.9) -- ++(30:0.3)
	(150:0.9) -- ++(150:0.3)
	(-30:0.9) -- ++(-30:-0.3)
	(-150:0.9) -- ++(-150:-0.3)
	(-90:0.9) -- ++(-90:0.3);
\node at (10,-1.4) {\small $(6.1)$};


\draw[xshift=12.5cm]
	(30:0.9) -- ++(30:0.3)
	(150:0.9) -- ++(150:0.3)
	(-30:0.9) -- ++(-30:-0.3)
	(-150:0.9) -- ++(-150:-0.3)
	(-90:0.9) -- ++(-90:-0.3);
\node at (12.5,-1.4) {\small $(6.2)$};


\draw[shift={(10cm,-2.6cm)}]
	(30:0.9) -- ++(30:-0.3)
	(150:0.9) -- ++(150:-0.3)
	(-30:0.9) -- ++(-30:0.3)
	(-150:0.9) -- ++(-150:0.3)
	(-90:0.9) -- ++(-90:0.3);
\node at (10,-4) {\small $(6.3)$};


\draw[shift={(12.5cm,-2.6cm)}]
	(30:0.9) -- ++(30:-0.3)
	(150:0.9) -- ++(150:-0.3)
	(-30:0.9) -- ++(-30:0.3)
	(-150:0.9) -- ++(-150:0.3)
	(-90:0.9) -- ++(-90:-0.3);
\node at (12.5,-4) {\small $(6.4)$};


\draw[shift={(10cm,-5.2cm)}]
	(30:0.9) -- ++(30:-0.3)
	(150:0.9) -- ++(150:0.3)
	(-30:0.9) -- ++(-30:0.3)
	(-150:0.9) -- ++(-150:-0.3)
	(-90:0.9) -- ++(-90:0.3);
\node at (10,-6.6) {\small $(6.5)$};
	

\draw[shift={(12.5cm,-5.2cm)}]
	(30:0.9) -- ++(30:-0.3)
	(150:0.9) -- ++(150:0.3)
	(-30:0.9) -- ++(-30:0.3)
	(-150:0.9) -- ++(-150:-0.3)
	(-90:0.9) -- ++(-90:-0.3);
\node at (12.5,-6.6) {\small $(6.6)$};

\end{tikzpicture}
\caption{Configurations from the identifications among $x_*,y_*,z_*$.}
\label{xyz}
\end{figure}
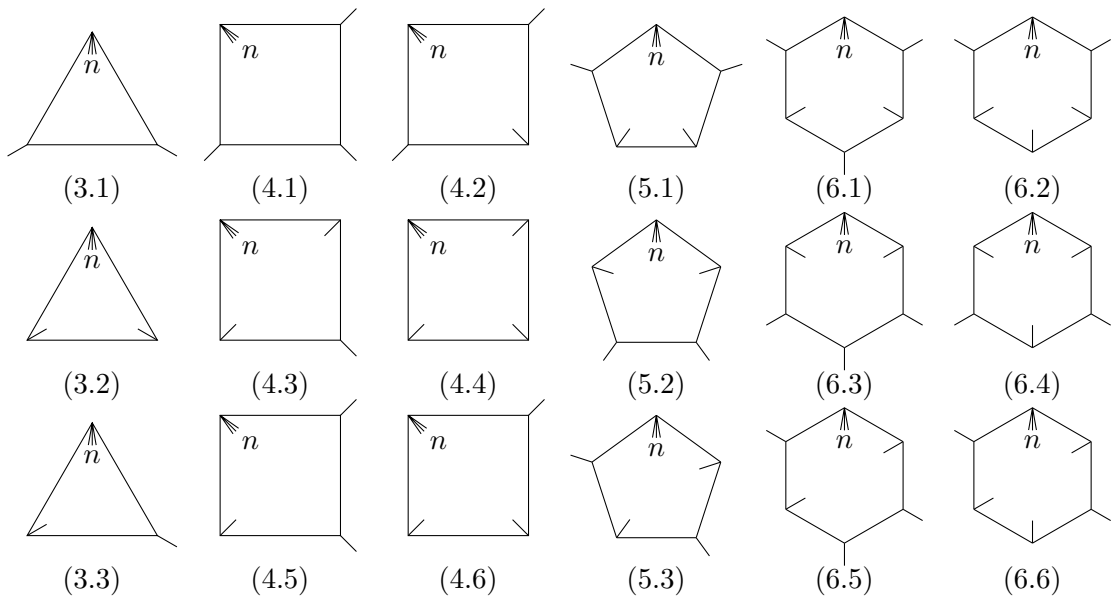

We know the following about the configurations in Figure \ref{xyz}. Any configuration is part of a combinatorial spherical tiling. The boundary has a unique high degree vertex. Although the interior allows high degree vertices, any vertex of distance $\le 3$ from the high degree boundary vertex has degree $3$. Note that although the theorem assumes degree $3$ within distance $4$ of the high degree boundary vertex, the stronger assumption of within distance $3$ is need for proving Theorem \ref{distance4}.  

By Figure \ref{impossible}, (3.1) and (4.1) are impossible. By considering the outside of the configurations, (3.2) and (4.4) are also impossible. Moreover, applying Lemma \ref{sequence} to $(6.3)$ gives $(3.1)$, which we have shown is impossible. Applying the lemma to the outside of $(6.2)$ gives similar contradiction. We conclude that the configurations $(3.1)$, $(3.2)$, $(4.1)$, $(4.4)$, $(6.2)$, $(6.3)$ are all impossible. For the remaining configurations, we follow the argument in the proof of Lemma \ref{cycle}. All the non-generic reductions can be dismissed as before, so that we only need to consider the generic reductions in Figure \ref{xyz_reduce}. A subtle point in constructing the reductions is to make sure that all the newly created vertices are of distance $\le 3$ from the high degree boundary vertex, so that the new vertices have degree $3$.

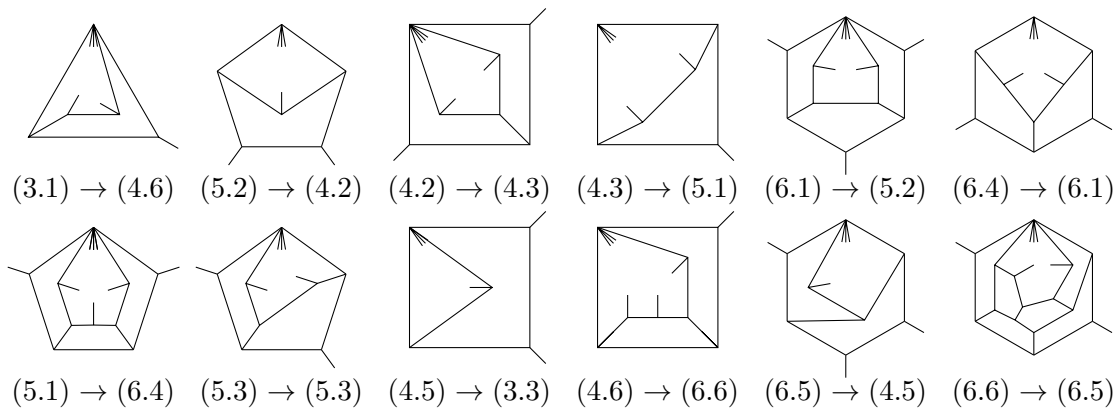
\begin{figure}[h]
\centering
\begin{tikzpicture}[>=latex,scale=1]



\draw[yshift=-0.2cm]
	(-30:1) -- (90:1) -- (210:1) -- cycle
	(90:1) -- ++(-80:0.3)
	(90:1) -- ++(-90:0.3) 
	(90:1) -- ++(-100:0.3)
	(-30:1) -- ++(-30:0.3)
	(90:1) -- (-30:0.4) -- (210:0.4) -- (210:1)
	(-30:0.4) -- ++(150:0.3)
	(210:0.4) -- ++(60:0.3);
\node at (0,-1.4) {\small $(3.1) \to (4.6)$};


\foreach \a / \b in {0/1, 1/0, 1/1}
{
\foreach \c in {0,...,4}
\draw[shift={(2.5*\a cm, -0.1 cm -2.7*\b cm)},rotate=72*\c]
	(18:0.9) -- (90:0.9);

\draw[shift={(2.5*\a cm, -2.7*\b cm)}]
	(90:0.8) -- ++(-80:0.3)
	(90:0.8) -- ++(-90:0.3) 
	(90:0.8) -- ++(-100:0.3);
}


\draw[yshift=-2.8cm]
	(90:0.9) -- ++(-80:0.3)
	(90:0.9) -- ++(-90:0.3) 
	(90:0.9) -- ++(-100:0.3)
	(18:0.9) -- ++(18:0.3)
	(162:0.9) -- ++(162:0.3)
	(90:0.9) -- (18:0.5) -- (-54:0.5) -- (-54:0.9)
	(18:0.5) -- (18:0.2)
	(90:0.9) -- (162:0.5) -- (234:0.5) -- (234:0.9)
	(162:0.5) -- (162:0.2)
	(-54:0.5) -- (-90:0.4) -- (234:0.5)
	(-90:0.4) -- (-90:0.1);
\node at (0,-4.1) {\small $(5.1) \to (6.4)$};


\draw[shift={(2.5cm,-0.1cm)}]
	(234:0.9) -- ++(234:0.3)
	(-54:0.9) -- ++(-54:0.3)
	(18:0.9) -- (-90:0.3) -- (162:0.9)
	(-90:0.3) -- (0,0);
\node at (2.5,-1.4) {\small $(5.2)\to (4.2)$};


\draw[shift={(2.5cm,-2.8cm)}]
	(162:0.9) -- ++(162:0.3)
	(-54:0.9) -- ++(-54:0.3)
	(18:0.9) -- (18:0.5) -- (234:0.5) -- (234:0.9)
	(18:0.5) -- ++(162:0.3)
	(90:0.9) -- (162:0.5) -- (234:0.5)
	(162:0.5) -- ++(162:-0.3);
\node at (2.5,-4.1) {\small $(5.3)\to (5.3)$};


\foreach \a in {2,3}
\foreach \b in {0,1}
\draw[shift={(2.5*\a cm,-2.7*\b cm)}]
	(-0.8,-0.8) rectangle (0.8,0.8)
	(-0.8,0.8) -- ++(-45:0.3)
	(-0.8,0.8) -- ++(-35:0.3)
	(-0.8,0.8) -- ++(-55:0.3);


\draw[xshift=5cm]
	(0.8,0.8) -- ++(45:0.3)
	(-0.8,-0.8) -- ++(45:-0.3)
	(0.8,-0.8) -- (0.4,-0.4)
	(-0.8,0.8) -- (0.4,0.4) -- (0.4,-0.4)
	(-0.8,0.8) -- (-0.4,-0.4) -- (0.4,-0.4)
	(0.4,0.4) -- ++(45:-0.3)
	(-0.4,-0.4) -- ++(45:0.3);
\node at (5,-1.4) {\small $(4.2) \to (4.3)$};


\draw[xshift=7.5cm]
	(0.8,-0.8) -- ++(-45:0.3)
	(0.8,0.8) -- (0.5,0.2) -- (-0.2,-0.5) -- (-0.8,-0.8) 
	(0.5,0.2) -- ++(-45:-0.3)
	(-0.2,-0.5) -- ++(-45:-0.3);
\node at (7.5,-1.4) {\small $(4.3) \to (5.1)$};


\draw[shift={(5cm,-2.7cm)}]
	(0.8,0.8) -- ++(45:0.3)
	(0.8,-0.8) -- ++(-45:0.3)
	(-0.8,-0.8) -- (0.3,0) -- (-0.8,0.8)
	(0.3,0) -- (0,0);
\node at (5,-4.1) {\small $(4.5) \to (3.3)$};
	

\draw[shift={(7.5cm,-2.7cm)}]
	(0.8,0.8) -- ++(45:0.3)
	(-0.8,-0.8) -- ++(45:0.3)
	(0.8,-0.8) -- ++(-45:-0.3)
	(0.8,-0.8) -- (0.4,-0.4) -- (-0.4,-0.4) -- (-0.8,-0.8)
	(0,-0.4) -- ++(0,0.3)
	(-0.4,-0.4) -- ++(0,0.3)
	(-0.8,0.8) -- (0.4,0.4) -- (0.4,-0.4)
	(0.4,0.4) -- ++(45:-0.3);
\node at (7.5,-4.1) {\small $(4.6) \to (6.6)$};


\foreach \a in {4,5}
\foreach \b in {0,1}
{
\foreach \c in {0,...,5}
\draw[shift={(2.5*\a cm,-2.7*\b cm)},rotate=60*\c]
	(-30:0.9) -- (30:0.9);

\draw[shift={(2.5*\a cm,-2.7*\b cm)}]
	(90:0.9) -- ++(-80:0.3)
	(90:0.9) -- ++(-90:0.3) 
	(90:0.9) -- ++(-100:0.3);
}


\draw[xshift=10cm]
	(30:0.9) -- ++(30:0.3)
	(150:0.9) -- ++(150:0.3)
	(-90:0.9) -- ++(-90:0.3)
	(90:0.9) -- (30:0.5) -- (-30:0.5) -- (-30:0.9)
	(90:0.9) -- (150:0.5) -- (210:0.5) -- (210:0.9)
	(-30:0.5) -- (210:0.5)
	(30:0.5) -- ++(190:0.3)
	(150:0.5) -- ++(-10:0.3);
\node at (10,-1.4) {\small $(6.1)\to (5.2)$};


\draw[xshift=12.5cm]
	(-30:0.9) -- ++(-30:0.3)
	(-150:0.9) -- ++(-150:0.3)
	(-90:0.9) -- (-90:0.5)
	(150:0.9) -- (180:0.4) -- (-90:0.5) -- (0:0.4) -- (30:0.9)
	(0:0.4) -- ++(150:0.3)
	(180:0.4) -- ++(30:0.3);
\node at (12.5,-1.4) {\small $(6.4)\to (6.1)$};


\draw[shift={(10cm,-2.7cm)}]
	(150:0.9) -- ++(150:0.3)
	(-30:0.9) -- ++(-30:0.3)
	(-90:0.9) -- ++(-90:0.3)
	(30:0.9) -- (-60:0.5) -- (210:0.9)
	(90:0.9) -- (180:0.5) -- (-60:0.5)
	(180:0.5) -- ++(10:0.3);
\node at (10,-4.1) {\small $(6.5)\to (4.5)$};
	

\draw[shift={(12.5cm,-2.7cm)}]
	(150:0.9) -- ++(150:0.3)
	(-30:0.9) -- ++(-30:0.3)
	(30:0.9) -- (-30:0.6) -- (-90:0.6) -- (-90:0.9)
	(90:0.9) -- (150:0.6) -- (210:0.6) -- (210:0.9)
	(210:0.6) -- (240:0.5) -- (-90:0.6)
	(240:0.5) -- (240:0.3)
	(150:0.6) -- (150:0.3) -- (240:0.3)
	(150:0.3) -- ++(30:0.3)
	(-30:0.6) -- (-30:0.3)
	(90:0.9) -- (30:0.6) -- (-30:0.3) -- (240:0.3)
	(30:0.6) -- ++(180:0.3);
\node at (12.5,-4.1) {\small $(6.6)\to (6.5)$};

\end{tikzpicture}
\caption{Generic reductions of configurations in Figure \ref{xyz}.}
\label{xyz_reduce}
\end{figure}

We also need to consider the configurations with small $n$ so that the generic reductions cannot be carried out. By Figure \ref{impossible}, the configurations $(3.1)_{n=0}$, $(4.2)_{n=0}$, $(4.5)_{n=0}$ are impossible. The reductions of the remaining special configurations are given by Figure \ref{xyz_special}.

\begin{figure}[h]
\centering
\begin{tikzpicture}[>=latex,scale=1]


\foreach \a in {0,1}
{
\draw[xshift=2.5*\a cm]
	(-0.8,-0.8) rectangle (0.8,0.8);
\node at (-0.9 cm+2.5*\a cm,0.9) {\small $\nu$};
}


\draw
	(0.8,0.8) -- ++(45:0.3)
	(-0.8,-0.8) -- ++(45:-0.3)
	(0.8,-0.8) -- (0.4,-0.4)
	(-0.8,0.8) -- (-0.4,0.4)
	(0,0) circle (0.566);
\node at (0,0) {\small $?$};
\node at (0,-1.2) {\small $(4.2)_{n=1}$};


\draw[xshift=2.5cm]
	(0.8,0.8) -- ++(45:0.3)
	(-0.8,-0.8) -- (0,0) -- (0.8,-0.8)
	(0,0) -- (0,-0.3);
\node at (2.5,-0.5) {\small $?$};
\node at (2.5,-1.2) {\small $(4.6)_{n=0}$};


\foreach \a in {2,3,4}
{
\foreach \c in {0,...,4}
\draw[shift={(2.5*\a cm, -0.1 cm)},rotate=72*\c]
	(18:0.9) -- (90:0.9);
\node at (2.5*\a cm,0.95) {\small $\nu$};
}


\draw[shift={(5cm,-0.1cm)}]
	(18:0.9) -- ++(18:0.3)
	(162:0.9) -- ++(162:0.3)
	(234:0.9) to[out=54,in=180] (0,-0.2) to[out=0,in=126] (-54:0.9);
\node at (5,-0.6) {\small $?$};
\node at (5,-1.2) {\small $(5.1)_{n=0}$};

\draw[shift={(7.5cm,-0.1cm)}]
	(18:0.9) -- ++(18:0.3)
	(162:0.9) -- ++(162:0.3)
	(90:0.9) -- (90:0.5) 
	(234:0.9) -- (234:0.5) -- (-90:0.4) -- (-54:0.5) -- (-54:0.9)
	(-90:0.4) -- (-90:0.1)
	(-54:0.5) arc (-54:234:0.5);
\node at (7.5,0.1) {\small $?$};
\node at (7.5,-1.2) {\small $(5.1)_{n=1}$};


\draw[shift={(10cm,-0.1cm)}]
	(162:0.9) -- ++(162:0.3)
	(-54:0.9) -- ++(-54:0.3)
	(18:0.9) -- (126:0.3) -- (234:0.9)
	(126:0.3) -- (0,0);
\node at (10.25,-0.4) {\small $?$};
\node at (10,-1.2) {\small $(5.3)_{n=0}$};


\begin{scope}[yshift=-2.8cm]

\foreach \a in {0,...,4}
{
\foreach \c in {0,...,5}
\draw[xshift=2.5*\a cm,rotate=60*\c]
	(-30:0.9) -- (30:0.9);
\node at (2.5*\a cm,1.05) {\small $\nu$};
}


\draw
	(30:0.9) -- ++(30:0.3)
	(150:0.9) -- ++(150:0.3)
	(-90:0.9) -- ++(-90:0.3)
	(210:0.9) to[out=30,in=150] (-30:0.9);
\node at (0,-1.5) {\small $(6.1)_{n=0}$};
\node at (0,-0.5) {\small $?$};

\draw[xshift=2.5cm]
	(-90:0.9) -- ++(-90:0.3)
	(30:0.9) -- ++(30:0.3)
	(150:0.9) -- ++(150:0.3)
	(90:0.9) -- (90:0.5)
	(-30:0.9) -- (-30:0.5)
	(-150:0.9) -- (-150:0.5)
	(0,0) circle (0.5);
\node at (2.5,-1.5) {\small $(6.1)_{n=1}$};
\node at (2.5,0) {\small $?$};


\draw[xshift=5cm]
	(150:0.9) -- ++(150:0.3)
	(-30:0.9) -- ++(-30:0.3)
	(-90:0.9) -- ++(-90:0.3)
	(30:0.9) -- (-60:0.5) -- (210:0.9)
	(-60:0.5) -- (-60:0.2);
\node at (5,-1.5) {\small $(6.5)_{n=0}$};
\node at (4.9,0.2) {\small $?$};
	

\draw[xshift=7.5cm]
	(150:0.9) -- ++(150:0.3)
	(-30:0.9) -- ++(-30:0.3)
	(30:0.9) -- (120:0.5) -- (210:0.9)
	(120:0.5) -- (0,0) -- (-90:0.9)
	(0,0) -- ++(30:-0.3);
\node at (7.5,-1.5) {\small $(6.6)_{n=0}$};
\node at (7.2,-0.4) {\small $?$};

\draw[xshift=10cm]
	(150:0.9) -- ++(150:0.3)
	(-30:0.9) -- ++(-30:0.3)
	(30:0.9) -- (30:0.5) -- (60:0.4) -- (90:0.5) -- (90:0.9)
	(210:0.9) -- (210:0.5) -- (240:0.4) -- (-90:0.5) -- (-90:0.9)
	(90:0.5) arc (90:210:0.5)
	(30:0.5) arc (30:-90:0.5)
	(60:0.4) -- (-30:0.2) -- (240:0.4)
	(-30:0.2) -- (150:0.1);
\node at (10,-1.5) {\small $(6.6)_{n=0}$};
\node at (9.8,0.1) {\small $?$};

\end{scope}

\end{tikzpicture}
\caption{Special reductions of configurations in Figure \ref{xyz}.}
\label{xyz_special}
\end{figure}
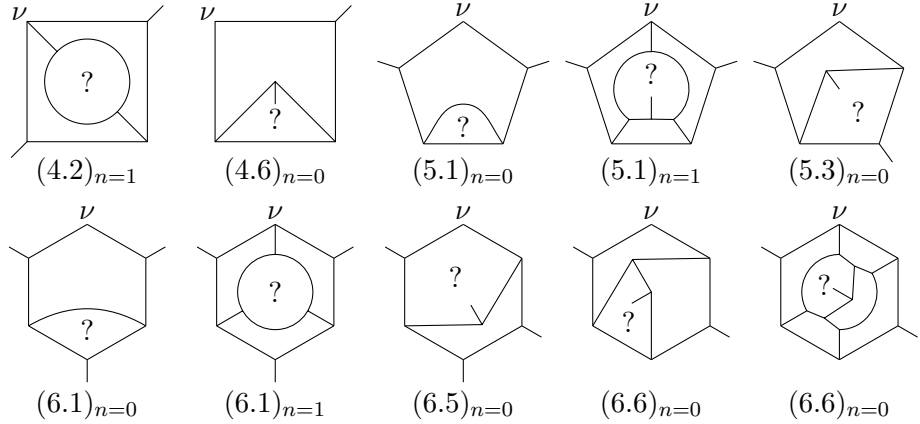

After proving that $x_*,y_*,z_*$ are all distinct, we may apply Lemma \ref{sequence} to construct all the second layer tiles, each having one $x_*$, two $y_*$ and two $z_*$ as the five vertices. The two layers of tiles are given by Figure \ref{map5}. 

Each degree $3$ vertex $z_*$ is connected to one $y_*$, another $z_*$, and a new vertex $w_*$. Since the distance between $w_*$ and $\nu$ are $\le 4$, $w_*$ have degree $3$. Applying Lemma \ref{sequence} to the boundary of the two layers, we find tiles with two $y_*$, two $z_*$ and one $w_*$ as the five vertices. For topological reason, $w_*$ cannot be identified with $x_*$. Since $y_*$ have degree $3$, $w_*$ cannot be identified with $y_*$.  If $z_i$ and $w_j$ are identified, then we get three possibilities in Figure \ref{zw}. By Figure \ref{impossible}, the second and third are impossible. Since the outside of the first is the second, the first is also impossible. Therefore $w_*$ cannot be identified with $z_*$.

\begin{figure}[htp]
\centering
\begin{tikzpicture}[>=latex,scale=1]

\foreach \a in {0,1,2}
{
\begin{scope}[xshift=5*\a cm]

\foreach \x in {-2,-1,1,2}
\draw[xshift=0.8*\x cm]
	(0,0.6) -- (0,0)
	(0,0) -- (0.2,-0.3) -- (0.2,-0.7) -- (-0.2,-0.7)
	(0,0) -- (-0.2,-0.3) -- (-0.2,-0.7);

\draw
	(-1.4,-0.3) -- (-1,-0.3)
	(1.4,-0.3) -- (1,-0.3)
	(-1.8,-0.3) -- (-2,-0.3)
	(1.8,-0.3) -- (2,-0.3);

\draw[dotted]
	(-0.4,0.2) -- (0.4,0.2);
\node at (0,0.75) {\small $\nu$};

\end{scope}
}


\draw 
	(-0.6,-0.7) -- ++(-60:0.2)
	(0.6,-0.7) -- ++(60:-0.2)
	(-1,-0.7) to[out=240,in=180] (0,-1.6) to[out=0,in=-60] (1,-0.7);
\node at (-1.6,-0.9) {\scriptsize $w_j=z_i$};
\node at (1.25,-0.9) {\scriptsize $z_j$};


\begin{scope}[xshift=5cm]

\draw
	(-0.6,-0.7) -- ++(-60:0.2)
	(0.6,-0.7) -- ++(60:-0.2)
	(-1,-0.7) -- ++(240:0.2)
	(1,-0.7) -- ++(-60:0.2)
	(-1.4,-0.7) to[out=-60,in=180] (0,-1.6) to[out=0,in=240] (1.4,-0.7);
\node at (-1.85,-0.9) {\scriptsize $w_j=z_i$};
\node at (1.5,-0.9) {\scriptsize $z_j$};	
\node at (0,-1.3) {\small ?};

\end{scope}


\begin{scope}[xshift=10cm]

\draw
	(-0.6,-0.7) -- ++(-60:0.2)
	(0.6,-0.7) -- ++(60:-0.2)
	(-1,-0.7) -- ++(60:-0.2)
	(1,-0.7) -- ++(-60:0.2)
	(-1.4,-0.7) -- ++(-60:0.2)
	(-1.8,-0.7) to[out=240,in=180] (-0.2,-1.6) to[out=0,in=240] (1.4,-0.7);
\node at (-2.4,-0.9) {\scriptsize $w_j=z_i$};
\node at (1.5,-0.9) {\scriptsize $z_j$};	
\node at (-0.2,-1.3) {\small ?};

\end{scope}

\end{tikzpicture}
\caption{Three cases of $w_i=z_j$.}
\label{zw}
\end{figure}
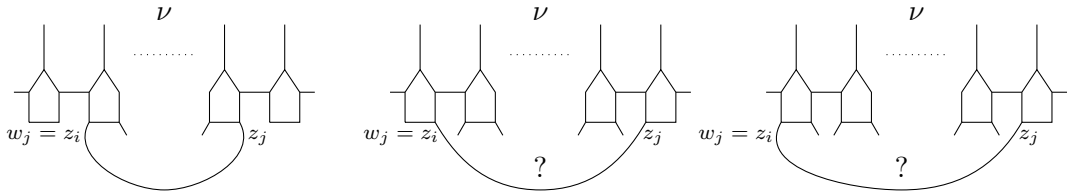

Finally, each $w_*$ is connected to two $z_*$. For $i\ne j$, $w_i$ and $w_j$ are connected to four distinct $z_*$. Since $w_i$ and $w_j$ have degree $3$, they cannot be identified. So we conclude that all $x_*,y_*,z_*,w_*$ are distinct. Then we may apply Lemma \ref{sequence} to construct the next layer of tiles. By Lemma \ref{sequence}, we find that the third edges at $w_*$ converge at the same vertex. Therefore we get the earth map tiling in Figure \ref{map5}.
\end{proof}

Theorem \ref{distance5} describes the case of maximal distance $5$ between any vertex of degree $>3$ and the neighboring vertices of degree $>3$. The following is the next extreme case.

\begin{theorem}\label{distance4}
Suppose the distances between vertices of degree $>3$ in a combinatorial pentagonal tiling are always $\ge 4$, then the tiling is either the earth map tiling in Figure \ref{map5}, or the earth map tiling in Figure \ref{map4}.
\end{theorem}

\begin{figure}[htp]
\centering
\begin{tikzpicture}[>=latex,scale=1]

\begin{scope}[rotate=10]
{
\foreach \x in {0,...,17}
\draw[rotate=20*\x]
		(0,0) -- (0:1)  -- (5:1.5) -- (15:1.5) -- (20:1)
		(5:1.5) -- (5:2) -- (-5:2) -- (-5:1.5);
		
\foreach \x in {0,...,5}
\draw[rotate=60*\x]
		(5:2) -- (10:2.3) -- (15:2)
		(25:2) -- (30:2.3) -- (35:2)
		(10:2.3) -- (20:2.6) -- (30:2.3)
		(20:2.6) -- (20:3)
		(-5:2) -- (-5:3)
		(45:2) -- (45:3);
}
\end{scope}

\begin{scope}[xshift=4cm]
{
\fill[gray!30]
	(-0.2,0.4) rectangle (0.2,-1.2);

\foreach \x in {0,...,6}
\draw[xshift=0.8*\x cm]
	(-0.2,0.4) -- (-0.2,0) -- (-0.4,0)
	(0.2,0.4) -- (0.2,0) -- (0.4,0)
	(-0.4,0.8) -- (-0.2,0.4) -- (0.2,0.4) -- (0.4,0.8)
	(0.4,0.8) -- (0.4,1.3);

\foreach \x in {0,1}
\draw[xshift=2.4*\x cm]
	(-0.2,0) -- (-0.2,-1.2)
	(0.2,0) -- (0.2,-1.2)
	(0.6,0) -- (1,-0.4) -- (1.2,-0.8) -- (1.4,-0.4) -- (1.8,0)
	(1,-0.4) -- (1,0) 
	(1.4,-0.4) -- (1.4,0)
	(1.2,-0.8) -- (1.2,-1.2);

\draw
	(-0.4,0.8) -- (-0.4,1.3)
	(4.6,0) -- (4.6,-1.2)
	(5,0) -- (5,-1.2);

\node at (2.4,1.5) {\small $\nu$};
\node at (0.3,-1.4) {\scriptsize $w_i=\sigma$};
\node at (0.65,-0.45) {\scriptsize $w_{i+1}$};
\node at (1.8,-0.45) {\scriptsize $w_{i+2}$};
\node at (1.33,-0.8) {\scriptsize $u$};
\node at (0.6,-0.9) {\scriptsize 1};
\node at (0.8,0.1) {\scriptsize 2};
\node at (1.2,-0.3) {\scriptsize 3};
\node at (1.6,0.1) {\scriptsize 4};
\node at (1.8,-0.9) {\scriptsize 5};
\node at (2.4,-0.6) {\scriptsize 6};
}
\end{scope}

\end{tikzpicture}
\caption{Earth map tiling, distance between poles $=4$.}
\label{map4}
\end{figure}
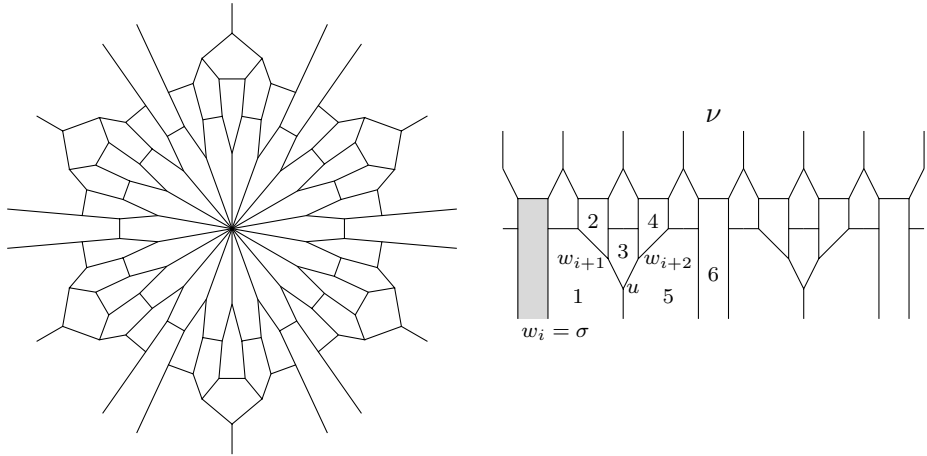

\begin{proof}
Theorem \ref{distance5} covers the case that the distances between vertices of degree $>3$ are always $\ge 5$. Moreover, under the assumption of the current theorem, the whole proof of Theorem \ref{distance5} except the last paragraph is still valid. We may still conclude that $x_*,y_*,z_*$ are all distinct, and $w_*$ is not identified with any of $x_*,y_*,z_*$. In fact, if we also know that all $w_*$ are distinct, then the tiling must be the earth map tiling in Figure \ref{map5}.

Now we assume some $w_i$ has degree $>3$. On the right of Figure \ref{map4}, the shaded region is one tile with $w_i$ as a vertex. We ``open up'' the tile at $w_i$ because $w_i$ will be the south pole $\sigma$. We assume that there are plenty of edges at $w_i=\sigma$.

We may construct the tile 1 by using Lemma \ref{sequence}. The tile makes use of the next $w$-vertex $w_{i+1}$ and creates another new vertex $u$. Since the distance between $w_{i+1}$, $u$ and $\sigma$ is $\le 2$, and the degree of $\sigma$ is $>3$, both $w_{i+1}$ and $u$ have degree $3$. Like the proof of Theorem \ref{distance5}, this implies $w_{i+1}$ cannot be identified with any other $w_*$ (because otherwise $w_{i+1}$ would be connected to four distinct $z_*$). If $u=z_j$, then $w_{i+1}$ is identified with another $w_*$, a contradiction. If $u=w_j$, then by the degree $3$ property of $w_{i+1}$ and $u$, we have the first picture in Figure \ref{w1}, which is impossible by Figure \ref{impossible}. Therefore the construction of the tile 1 must be generic, as illustrated in Figure \ref{map4}.

After the tile 1, we immediately have the tiles 2 and 3 by Lemma \ref{sequence}. Then the distance between $w_{i+2}$ and $\sigma$ is $2$. By the same reason as before, the degree of $w_{i+2}$ must be $3$, and the construction of the tiles must be generic. Then by Lemma \ref{sequence}, we further get the tiles 4, 5 and 6. 

Now the tile 6 is in the same configuration as the shaded tile we started with. Therefore the process repeats itself. The process stops when the edges at $\sigma$ are used up. This means that either we come back to the shaded tile in a perfect match, which gives the earth map tiling, or there are not enough available edges at $\sigma$ to complete the last ``gap''. The second and third pictures in Figure \ref{w1} show what happens when there are not enough available edges. Both are impossible by Figure \ref{impossible}.
\end{proof}

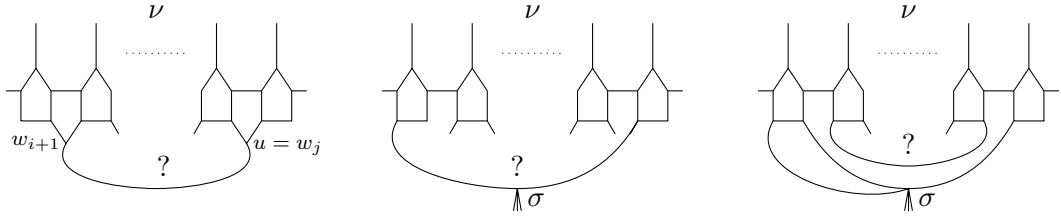
\begin{figure}[htp]
\centering
\begin{tikzpicture}[>=latex,scale=1]

\foreach \a in {0,1,2}
{
\begin{scope}[xshift=5*\a cm]

\foreach \x in {-2,-1,1,2}
\draw[xshift=0.8*\x cm]
	(0,0.6) -- (0,0)
	(0,0) -- (0.2,-0.3) -- (0.2,-0.7) -- (-0.2,-0.7)
	(0,0) -- (-0.2,-0.3) -- (-0.2,-0.7);

\draw
	(-1.4,-0.3) -- (-1,-0.3)
	(1.4,-0.3) -- (1,-0.3)
	(-1.8,-0.3) -- (-2,-0.3)
	(1.8,-0.3) -- (2,-0.3);

\draw[dotted]
	(-0.4,0.2) -- (0.4,0.2);
\node at (0,0.75) {\small $\nu$};

\end{scope}
}


\draw 
	(-0.6,-0.7) -- ++(-60:0.2)
	(0.6,-0.7) -- ++(60:-0.2)
	(-1,-0.7) -- (-1.2,-1) -- (-1.4,-0.7)
	(1,-0.7) -- (1.2,-1) -- (1.4,-0.7)
	(-1.2,-1) to[out=240,in=180] (0,-1.6) to[out=0,in=-60] (1.2,-1);
\node at (-1.6,-1) {\scriptsize $w_{i+1}$};
\node at (1.75,-1.05) {\scriptsize $u=w_j$};
\node at (0.1,-1.3) {\small ?};


\begin{scope}[xshift=5cm]

\draw
	(-0.6,-0.7) -- ++(-60:0.2)
	(0.6,-0.7) -- ++(60:-0.2)
	(-1,-0.7) -- ++(60:-0.2)
	(1,-0.7) -- ++(-60:0.2)
	(1.4,-0.7) -- ++(60:-0.2)
	(-0.2,-1.6) -- ++(-80:0.3)
	(-0.2,-1.6) -- ++(-90:0.3) 
	(0-.2,-1.6) -- ++(-100:0.3)
	(-1.8,-0.7) to[out=240,in=180] (-0.2,-1.6) to[out=0,in=240] (1.4,-0.7);	
\node at (-0.2,-1.3) {\small ?};
\node at (0.05,-1.75) {\small $\sigma$};

\end{scope}


\begin{scope}[xshift=10cm]

\draw
	(-0.6,-0.7) -- ++(-60:0.2)
	(0.6,-0.7) -- ++(60:-0.2)
	(0,-1.6) -- ++(-80:0.3)
	(0,-1.6) -- ++(-90:0.3) 
	(0,-1.6) -- ++(-100:0.3)
	(-1.4,-0.7) to[out=-60,in=180] (0,-1.6) to[out=0,in=240] (1.4,-0.7)
	(-1,-0.7) to[out=240,in=180] (0,-1.3) to[out=0,in=-60] (1,-0.7)
	(-1.8,-0.7) to[out=240,in=200] (0,-1.6);
\node at (0,-1) {\small ?};
\node at (0.25,-1.75) {\small $\sigma$};

\end{scope}

\end{tikzpicture}
\caption{The new tiles are generic; and what happens when there are not enough available edges at $\sigma$.}
\label{w1}
\end{figure}

\section{Two Vertices of Degree $>3$}

Theorems \ref{distance5} and \ref{distance4} cannot be further extended to smaller distances between vertices of degree $>3$. The connected sum construction in the introduction produces many vertices of degree $>3$ within distance $1$. Figure \ref{d3} (with four outward rays converging at the same vertex) is a combinatorial pentagonal tiling of the sphere with three vertices of degree $4$ and eighteen tiles, such that the distance between any two vertices of degree $4$ is $3$.

\begin{figure}[htp]
\centering
\begin{tikzpicture}[>=latex,scale=1]

\foreach \x in {1,-1}
\foreach \y in {1,-1}
\draw[xscale=\x,yscale=\y]
	(0,1.1) -- (0.4,1.1) -- (1.2,1.2) -- (1.1,0.4) -- (1.1,0)
	(0,0.6) -- (0.4,0.8) -- (0.8,0.4) -- (0.6,0)
	(0.4,0.8) -- (0.4,1.1)
	(0.8,0.4) -- (1.1,0.4)
	(1.2,1.2) -- (1.5,1.5);

\foreach \x in {1,-1}
\draw[scale=\x]
	(0.4,0.8) -- (0.2,0.2) -- (0,0)
	(0,-0.6) -- (0.3,-0.1)
	(0.2,0.2) -- (0.3,-0.1) -- (0.6,0);

\end{tikzpicture}
\caption{Three vertices of degree $4$ and eighteen tiles.}
\label{d3}
\end{figure}
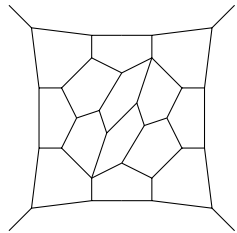

In view of Theorem \ref{1vertex}, therefore, we may study combinatorial pentagonal tilings of the sphere with exactly two vertices of degree $>3$. We expect to get similar earth map tilings with the two high degree vertices as the poles, even when the distance between the two poles is $\le 3$. We call the shortest paths connecting the poles \emph{meridians}. If the meridians have length $5$ or $4$, then the tilings have to be the earth map tilings in Figures \ref{map5} and \ref{map4}. So we only need to consider meridians of length $\le 3$. Up to the horizontal and vertical flippings, Figure \ref{allcase2} gives all the possibilities.

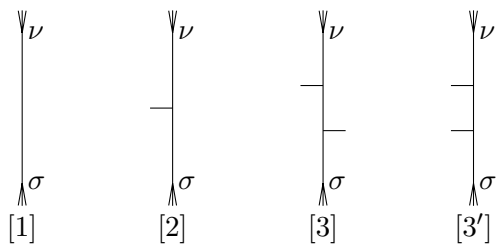
\begin{figure}[h]
\centering
\begin{tikzpicture}[>=latex,scale=1]

\foreach \x in {0,...,3}
\draw[xshift=2*\x cm]
	(0,-1) node[right=-2] {\small $\sigma$} -- (0,1) node[right=-2] {\small $\nu$}
	(0,1) -- ++(90:0.3)
	(0,1) -- ++(80:0.3) 
	(0,1) -- ++(100:0.3)
	(0,-1) -- ++(90:-0.3)
	(0,-1) -- ++(80:-0.3) 
	(0,-1) -- ++(100:-0.3);

\node at (0,-1.6) {\small [1]};

\draw[xshift=2cm] 
	(0,0) -- (-0.3,0);
\node[xshift=2cm] at (0,-1.6) {\small $[2]$};

\draw[xshift=4cm] 
	(0,0.3) -- (-0.3,0.3) 
	(0,-0.3) -- (0.3,-0.3);
\node[xshift=4cm] at (0,-1.6) {\small $[3]$};

\draw[xshift=6cm] 
	(0,0.3) -- (-0.3,0.3) 
	(0,-0.3) -- (-0.3,-0.3);
\node[xshift=6cm] at (0,-1.6) {\small $[3']$};

\end{tikzpicture}
\caption{Possible meridians of length $\le 3$.}
\label{allcase2}
\end{figure}

We put any meridian in the vertical position and develop the tiling on the right side (or the east side).  We find that the development comes back to the original meridian we started with. Therefore the same development can repeat itself. We call the tiling between adjacent meridians \emph{timezones}. By glueing several timezones bounded by the same meridians, we get the earth map tilings.

Figure \ref{map1} gives the earth map tiling with the meridian $[1]$ (indicated by the thick line), in both the polar and equator views. Figure \ref{map2} gives the earth map tiling with the meridian $[2]$, which also includes the ``meridian'' $[3']$ (so $[3']$ cannot be a meridian, see proof below). Figure \ref{map3} gives the earth map tiling with the meridian $[3]$.

\begin{figure}[h]
\centering
\begin{tikzpicture}[>=latex,scale=1]


\foreach \x in {0,...,5}
{
\draw[rotate=30+60*\x]
	(0,0) -- (-15:0.6)  -- (-5:1.2) -- (5:1.2) -- (15:0.6) -- cycle
	(-5:1.2) -- (-5:1.6) -- (0:1.8) -- (5:1.6) -- (5:1.2)
	(5:1.6) -- (10:1.6) -- (20:1.2) -- (15:0.6)
	(-5:1.6) -- (-10:1.6) -- (-20:1.2) -- (-15:0.6)
	(10:1.6) -- (10:2.1) -- (0:2) -- (0:1.8)
	(-10:1.6) -- (-10:2.1) -- (0:2) -- (0:1.8)
	(10:2.1) -- (12:2.3) -- (-12:2.3) -- (-10:2.1)
	(20:2.6) -- (12:2.3)
	(20:1.2) -- (20:3)
	(-20:2.6) -- (-12:2.3)
	(-20:1.2) -- (-20:3);
\draw[very thick, rotate=60*\x]
	(0,0) -- (0:3);
}


\foreach \a in {0,1,2}
{
\begin{scope}[xshift=5cm+2*\a cm]

\foreach \x in {1,-1}
\foreach \y in {1,-1}
{
\draw[xscale=\x,yscale=\y]
	(0,0) -- (0,0.2) -- (0.3,0.3) -- (0.5,0) -- (0.7,0) -- (0.7,1)
	(0.3,0.3) -- (0.3,0.6) -- (0.7,0.7)
	(0,0.6) -- (0.3,0.6);
\draw[xscale=\x,yscale=\y,very thick]
	(1,0) -- (1,1);
}

\end{scope}
}

\node at (7,1.2) {\small $\nu$};
\node at (7,-1.2) {\small $\sigma$};

\end{tikzpicture}
\caption{Earth map tiling, distance between poles $=1$.}
\label{map1}
\end{figure}
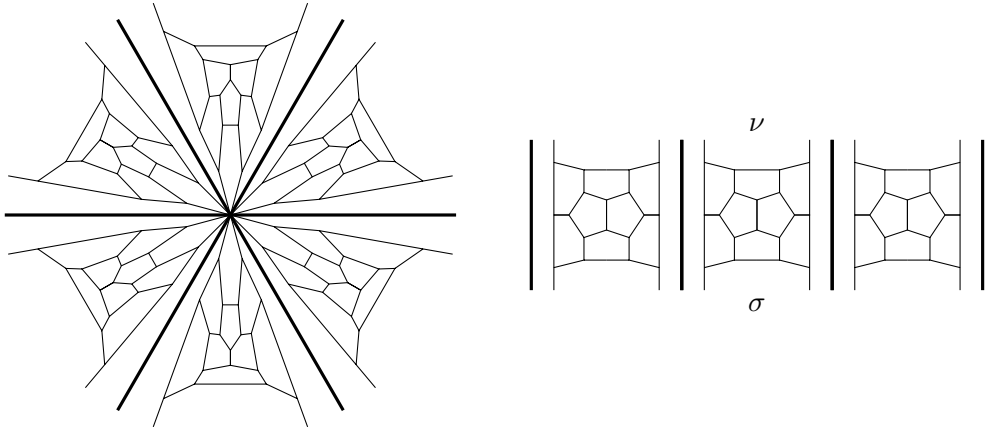

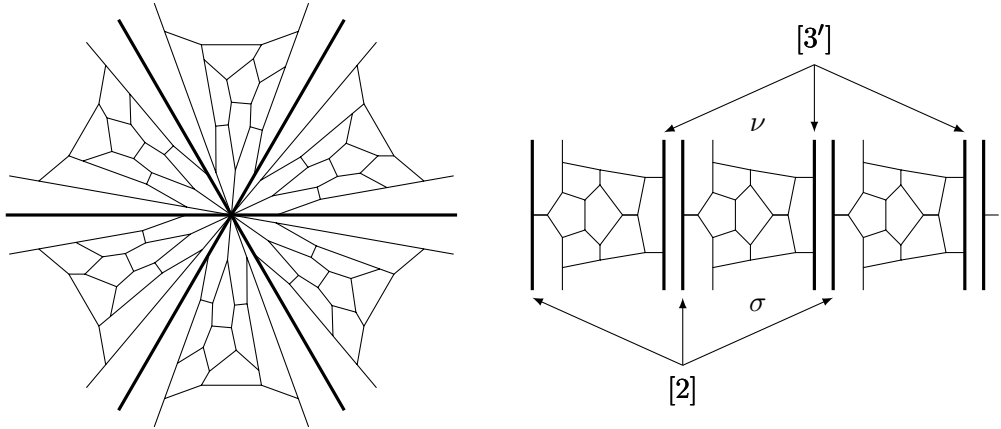
\begin{figure}[h]
\centering
\begin{tikzpicture}[>=latex,scale=1]


\foreach \x in {0,...,5}
{
\draw[rotate=60*\x]
	(0,0) -- (0:0.6) -- (10:1.2) -- (18:1.2) -- (25:0.6)
	(0,0) -- (25:0.6) -- (32:1.2) -- (40:1.2) -- (50:0.6)
	(10:1.2) -- (10:3)
	(50:0) -- (50:3)
	(18:1.2) -- (20:1.5) -- (30:1.5) -- (32:1.2)
	(20:1.5) -- (18:1.8) 
	(32:1.8) -- (30:1.5)
	(10:2.1) -- (18:1.8) -- (22:2.1) -- (32:1.8) -- (40:2.1)
	(22:2.1) -- (20:2.3)
	(10:2.6) -- (20:2.3) -- (40:2.3)
	(40:1.2) -- (40:2.3) -- (50:2.6)
	;
\draw[very thick, rotate=60*\x]
	(0,0) -- (0:3);
}


\foreach \a in {0,1,2}
{
\begin{scope}[xshift=5cm+2*\a cm]

\foreach \y in {1,-1}
{
\draw[yscale=\y]
	(-1,0) -- (-0.8,0) -- (-0.6,0.3) -- (-0.6,1)
	(-0.6,0.3) -- (-0.3,0.2) -- (-0.1,0.4) -- (0.2,0) -- (0.4,0)
	(-0.3,0.2) -- (-0.3,0)
	(-0.6,0.7) -- (-0.1,0.6) -- (-0.1,0.4)
	(-0.1,0.6) -- (0.5,0.5)
	(0.4,0) -- (0.5,0.5) -- (0.75,0.5);
\draw[yscale=\y,very thick]
	(1,0) -- (1,1)
	(0.75,0) -- (0.75,1);
}

\end{scope}
}

\draw[xshift=5cm, very thick]
	(-1,-1) -- (-1,1);
\draw
	(10,0) -- ++(0.2,0);

\node at (7,1.2) {\small $\nu$};
\node at (7,-1.2) {\small $\sigma$};

\foreach \x in {-1,0,1}
{
\draw[->]
	(7.75,2) node[above] {\small $[3']$} -- (7.75+2*\x,1.1);
\draw[->]
	(6,-2) node[below] {\small $[2]$} -- (6+2*\x,-1.1);
}

\end{tikzpicture}
\caption{Earth map tiling, distance between poles $=2$.}
\label{map2}
\end{figure}

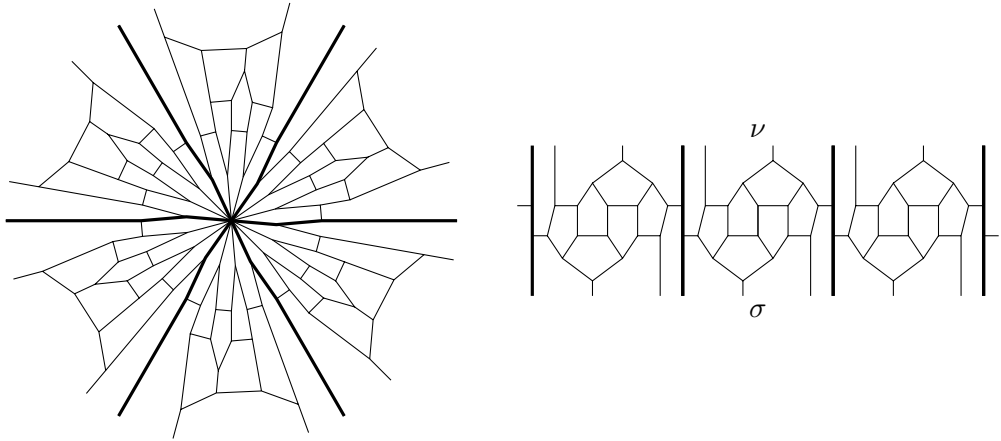
\begin{figure}[h]
\centering
\begin{tikzpicture}[>=latex,scale=1]


\foreach \x in {0,...,17}
\draw[rotate=15+20*\x]
		(0,0) -- (0:0.6)  -- (5:1.2) -- (15:1.2) -- (20:0.6);
		
\foreach \x in {0,...,5}
{
\draw[rotate=15+60*\x]
		(5:1.2) -- (5:1.6) -- (-5:1.6) -- (-5:1.2)
		(25:1.2) -- (25:1.6) -- (15:1.6) -- (15:1.2)
		(45:1.2) -- (45:3)
		(35:1.2) -- (35:3)
		(5:1.6) -- (10:2) -- (15:1.6)
		(10:2) -- (10:2.3) 
		(0:2.6) -- (10:2.3) -- (25:2.3) -- (35:2.6)
		(-5:1.6) -- (0:2.6) -- (0:3)
		(25:1.6) -- (25:2.3);
\draw[rotate=60*\x,very thick]	
	(0,0) -- (-5:0.6) -- (0:1.2) -- (0:3);
}


\foreach \a in {0,1,2}
{
\begin{scope}[xshift=5cm+2*\a cm]

\foreach \y in {1,-1}
{
\draw[scale=\y]
	(-1,-0.2) -- (-0.8,-0.2) -- (-0.7,0.2) -- (-0.7,1)
	(-0.7,0.2) -- (-0.4,0.2) -- (-0.4,-0.2) -- (-0.6,-0.5) -- (-0.8,-0.2)
	(-0.4,0.2) -- (-0.2,0.5) -- (0,0.2) -- (0,-0.2) -- (-0.4,-0.2)
	(-0.2,0.5) -- (0.2,0.8) -- (0.6,0.5)
	(0.2,0.8) -- (0.2,1);
\draw[yscale=\y,very thick]
	(1,-1) -- (1,1);
}

\end{scope}
}
\draw[xshift=5cm, very thick]
	(-1,-1) -- (-1,1);
\draw
	(10,-0.2) -- ++(0.2,0)
	(4,0.2) -- ++(-0.2,0);

\node at (7,1.2) {\small $\nu$};
\node at (7,-1.2) {\small $\sigma$};

\end{tikzpicture}
\caption{Earth map tiling, distance between poles $=3$.}
\label{map3}
\end{figure}

\begin{theorem}\label{2vertex}
There are five families of combinatorial pentagonal tilings of the sphere with exactly two vertices of degree $>3$. The five families are given by Figures \ref{map5}, \ref{map4}, \ref{map1}, \ref{map2}, \ref{map3}.
\end{theorem}

\begin{proof}
We need to argue three things. The first is what happens to the meridian $[3']$. The second is to justify the developments of the tilings in Figures \ref{map1}, \ref{map2}, \ref{map3} by dismissing the many possible non-generic cases. The third is to study the special cases that, after successively developing many timezones, there may not be enough available edges at the two poles to allow further development. 

First, the meridian $[3']$ is created at the end of developing the right side of the meridian $[2]$. If we start with $[3']$ and carry out the development, then we immediately get the meridian $[2]$ on the right side and furthermore the same tiles in the equator view in Figure \ref{map2}. So $[2]$ and $[3']$ give the same tiling. In particular, $[3']$ cannot be the shortest path and therefore cannot be the meridian.

Second, the non-generic developments can be dismissed for the reason largely similar to the proof of Lemma \ref{cycle}. The idea is that, after using Lemma \ref{sequence} to create a new tile, the identification of a newly created vertex and an existing vertex produces a simple closed path. If the length of the path is $\le 7$ and one of the regions divided by the path does not contain vertices of degree $>3$ in the interior, then Lemma \ref{cycle} can be applied to the region to dismiss the case.

For example, we start with the meridian $[3]$ and develop the tiles in Figure \ref{zone}. In the first picture, the shaded tile is newly created, together with a new vertex indicated by the dot. If the dot is identified with the south pole $\sigma$, then we get the second picture, and we can apply Lemma \ref{cycle} to the ``obvious'' simple closed path of length $6$.

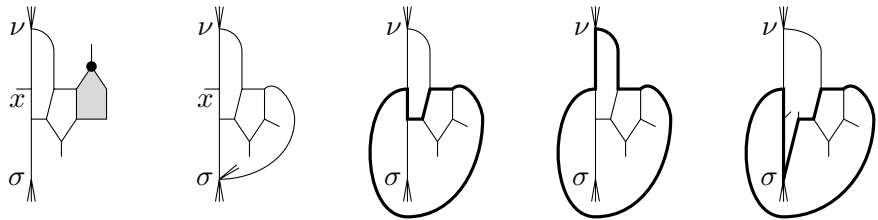
\begin{figure}[h]
\centering
\begin{tikzpicture}[>=latex,scale=1]

\fill[gray!30]
	(-0.2,0.5) -- (-0.4,0.2) -- (-0.4,-0.2) -- (0,-0.2) -- (0,0.2) -- cycle;

\draw
	(-1,-1) -- ++(-90:0.3)
	(-1,-1) -- ++(-80:0.3)
	(-1,-1) -- ++(-100:0.3)
	(-1,1) -- ++(90:0.3)
	(-1,1) -- ++(80:0.3)
	(-1,1) -- ++(100:0.3)
	(-1,-0.2) -- (-0.8,-0.2) -- (-0.7,0.2) -- (-0.7,0.7) arc (0:90:0.3)
	(-0.7,0.2) -- (-0.4,0.2) -- (-0.4,-0.2) -- (-0.6,-0.5) -- (-0.8,-0.2)
	(-0.4,0.2) -- (-0.2,0.5) -- (0,0.2) -- (0,-0.2) -- (-0.4,-0.2)
	(-1,0.2) node[below left=-2] {\small $x$} -- ++(-0.2,0)
	(-0.2,0.5) -- ++(0,0.3)
	(-0.6,-0.5) -- ++(0,-0.2)
	(-1,-1) node[left=-2] {\small $\sigma$} -- (-1,1) node[left=-2] {\small $\nu$} ;
\fill
	(-0.2,0.5) circle (0.07);

\draw[xshift=2.5cm]
	(-1,-1) -- ++(-90:0.3)
	(-1,-1) -- ++(-80:0.3)
	(-1,-1) -- ++(-100:0.3)
	(-1,1) -- ++(90:0.3)
	(-1,1) -- ++(80:0.3)
	(-1,1) -- ++(100:0.3)
	(-1,-1) -- ++(30:0.3)
	(-1,-1) -- ++(40:0.3)
	(-1,-0.2) -- (-0.8,-0.2) -- (-0.7,0.2) -- (-0.7,0.7) arc (0:90:0.3)
	(-0.7,0.2) -- (-0.4,0.2) -- (-0.4,-0.2) -- (-0.6,-0.5) -- (-0.8,-0.2)
	(-1,0.2) node[below left=-2] {\small $x$} -- ++(-0.2,0)
	(-0.6,-0.5) -- ++(0,-0.2)
	(-1,-1) node[left=-2] {\small $\sigma$} -- (-1,1) node[left=-2] {\small $\nu$}
	(-0.4,-0.2) -- ++(0.2,-0.1);
\draw[xshift=2.5cm]
	(-0.4,0.2) to[out=45,in=90] (0,-0.2) to[out=-90,in=0] (-1,-1);

\draw[xshift=5cm]
	(-1,-1) -- ++(-90:0.3)
	(-1,-1) -- ++(-80:0.3)
	(-1,-1) -- ++(-100:0.3)
	(-1,1) -- ++(90:0.3)
	(-1,1) -- ++(80:0.3)
	(-1,1) -- ++(100:0.3)
	(-1,-0.2) -- (-0.8,-0.2) -- (-0.7,0.2) -- (-0.7,0.7) arc (0:90:0.3)
	(-0.7,0.2) -- (-0.4,0.2) -- (-0.4,-0.2) -- (-0.6,-0.5) -- (-0.8,-0.2)
	(-0.6,-0.5) -- ++(0,-0.2)
	(-1,-1) node[left=-2] {\small $\sigma$} -- (-1,1) node[left=-2] {\small $\nu$}
	(-0.4,-0.2) -- ++(0.2,-0.1);
\draw[xshift=5cm,very thick]
	(-0.4,0.2) to[out=45,in=90] (0,-0.2) to[out=-90,in=0] (-1,-1.5) to[out=180,in=180] (-1,0.2) -- (-1,-0.2) -- (-0.8,-0.2) -- (-0.7,0.2) -- cycle;

\draw[xshift=7.5cm]
	(-1,-1) -- ++(-90:0.3)
	(-1,-1) -- ++(-80:0.3)
	(-1,-1) -- ++(-100:0.3)
	(-1,1) -- ++(90:0.3)
	(-1,1) -- ++(80:0.3)
	(-1,1) -- ++(100:0.3)
	(-1,-0.2) -- (-0.8,-0.2) -- (-0.7,0.2) -- (-0.7,0.7) arc (0:90:0.3)
	(-0.7,0.2) -- (-0.4,0.2) -- (-0.4,-0.2) -- (-0.6,-0.5) -- (-0.8,-0.2)
	(-0.6,-0.5) -- ++(0,-0.2)
	(-1,-1) node[left=-2] {\small $\sigma$} -- (-1,1) node[left=-2] {\small $\nu$}
	(-0.4,-0.2) -- ++(0.2,-0.1);
\draw[xshift=7.5cm,very thick]
	(-0.4,0.2) to[out=45,in=90] (0,-0.2) to[out=-90,in=0] (-1,-1.5) to[out=180,in=180] (-1,0.2) -- (-1,1) arc (90:0:0.3) -- (-0.7,0.2) -- cycle;
	
\draw[xshift=10cm]
	(-1,-1) -- ++(-90:0.3)
	(-1,-1) -- ++(-80:0.3)
	(-1,-1) -- ++(-100:0.3)
	(-1,1) -- ++(90:0.3)
	(-1,1) -- ++(80:0.3)
	(-1,1) -- ++(100:0.3)
	(-0.8,-0.2) -- (-0.6,-0.2) -- (-0.5,0.2) -- (-0.5,0.7) arc (0:90:0.5 and 0.3)
	(-0.5,0.2) -- (-0.2,0.2) -- (-0.2,-0.2) -- (-0.4,-0.5) -- (-0.6,-0.2)
	(-0.4,-0.5) -- ++(0,-0.2)
	(-1,-1) node[left=-2] {\small $\sigma$} -- (-1,1) node[left=-2] {\small $\nu$}
	(-1,-0.2) -- ++(0.1,0.1)
	(-0.8,-0.2) -- ++(0,0.1)
	(-0.2,-0.2) -- ++(0.2,-0.1);
\draw[xshift=10cm,very thick]
	(-0.2,0.2) to[out=45,in=90] (0.2,-0.2) to[out=-90,in=0] (-1,-1.5) to[out=180,in=180] (-1,0.2) -- (-1,-1) -- (-0.8,-0.2) -- (-0.6,-0.2) -- (-0.5,0.2) -- cycle;

\end{tikzpicture}
\caption{Non-generic developments starting from $[3]$.}
\label{zone}
\end{figure}

If the dot is identified with $x$, then we get the third picture. However, Lemma \ref{cycle} can be applied neither to the inside nor to the outside of the thick simple closed path, because either side contains a vertex of degree $>3$ in the interior. To resolve the problem, we may either change the simple closed path to the one in the fourth picture and apply Lemma \ref{cycle} to the outside, or we may modify the simple closed path in the third picture to the one in the fifth picture and apply Lemma \ref{cycle}.

With suitable choice of the sequence of constructing tiles and extra care, we can indeed show that, starting from any meridian, the developments to the right side must be generic.

\begin{figure}[h]
\centering
\begin{tikzpicture}[>=latex,scale=1]

\foreach \a in {0,3}
{
\foreach \x in {0,...,5}
\draw[xshift=\a cm, rotate=60*\x]
	(30:0.9) -- (90:0.9);
	
\draw[xshift=\a cm]
	(0,0.9) -- ++(-90:0.3)
	(0,0.9) -- ++(-80:0.3) 
	(0,0.9) -- ++(-100:0.3)
	(30:0.9) -- ++(30:-0.3)
	(150:0.9) -- ++(150:0.3)
	(-30:0.9) -- ++(-30:0.3)
	(210:0.9) -- ++(210:-0.3);
}

\draw
	(0,-0.9) -- (0,-1.2)
	(3,-0.9) -- (3,-0.6);

\end{tikzpicture}
\caption{Failing cases for the meridian $[3]$.}
\label{special2}
\end{figure}
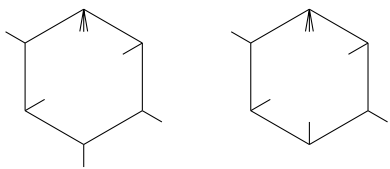

Now we deal with the third thing. We start with any meridian and add successive timezones. As long as we have at least two edges available at each pole, we may continue adding timezone. If there are exactly two edges available at each poles, then we can complete the final timezone and get the earth map tiling. This fails exactly when we have no or one edge left in one pole and at least one edge left at the other pole. Figure \ref{special2} describes the failing cases for the meridian $[3]$. Applying Lemma \ref{cycle}, these failing cases cannot happen. The same reason applies to the other meridians.
\end{proof}

\subsection*{Acknowledgements}
Part of the work is based on a HKUST UROP project by S. L. Wang and Y. Zhou. I would like to thank them for their contributions.

\end{document}